\documentclass[reqno]{amsart}
\usepackage{amsmath,amsbsy,amsfonts,enumerate,amssymb}
\usepackage{color,enumitem,graphicx}
\usepackage[colorlinks=true,urlcolor=blue,
citecolor=red,linkcolor=blue,linktocpage,pdfpagelabels,
bookmarksnumbered,bookmarksopen]{hyperref}
\usepackage{textgreek}

\date{}
\def\dis{\displaystyle}
\def\nd{\noindent}
\def\thend{\rule{3mm}{3mm}}
\def\Re{\mathbb{R}}
\newtheorem{theorem}{Theorem}[section]
\newtheorem{cor}{Corollary}[section]

\newtheorem{lem}{Lemma}[section]
\newtheorem{rmk}{Remark}[section]

\begin{document}

\title[Ground states of elliptic problems over cones]{Ground states of elliptic problems over cones}
\vspace{1cm}

\author{Giovany M. Figueiredo}
\address{G. M. Figueiredo \newline  Universidade de Brasília, Departamento de Matemática, 70910-900, Brasília, DF, Brazil}
\email{\tt giovany@unb.br} 

\author{Humberto Ramos Quoirin}
\address{H. Ramos Quoirin \newline CIEM-FaMAF, Universidad Nacional de C\'{o}rdoba, (5000)
	C\'{o}rdoba, Argentina}
\email{\tt humbertorq@gmail.com}

\author{Kaye Silva}
\address{K. Silva \newline Instituto de Matemática e Estatística
	Universidade Federal de Goiás,
	74001-970, Goiânia, GO, Brazil}
\email {\tt kayesilva@ufg.br} 

\thanks{The first author was partially supported by CNPq/Brazil under the grants [304657/2018-2] and [407479/2018-0], and FAPDf  04/2017. The third author was partially supported by CNPq/Brazil under the grant [408604/2018-2]}

\begin{abstract}
Given a reflexive Banach space $X$, we consider a class of functionals $\Phi \in C^1(X,\Re)$ that do not behave in a uniform way, in the sense that the map $t \mapsto \Phi(tu)$, $t>0$, does not have a uniform geometry with respect to $u\in X$. Assuming instead such a uniform behavior within an open cone  $Y \subset X \setminus \{0\}$, we show that $\Phi$ has a ground state relative to $Y$.
Some further conditions ensure that this relative ground state is the (absolute) ground state of $\Phi$.  Several applications to elliptic equations and systems are given.
\end{abstract}

\maketitle

\tableofcontents

\section{Introduction}
\medskip

It is widely known that critical point theory and variational techniques play a major role in a host of problems arising in life sciences. Over the last decades, a strong development of these methods and their applications has been witnessed in connection with the qualitative study of differential equations, for these ones often admit a weak formulation as a critical point equation
\begin{equation}\label{e}
\Phi'(u)=0.
\end{equation}
Here $\Phi'$ is the Fr\'echet derivative of a $C^1$ functional $\Phi$, the so-called {\it energy functional}, defined on a Banach space $X$. In view of the physical motivation behind some of these problems, a special attention has been paid to {\it ground state} solutions of \eqref{e}, i.e. those minimizing $\Phi$ over the set of nontrivial solutions. 

The most succesful approach to show the existence of such solutions consists in minimizing $\Phi$ over the set $$\mathcal{N}=\mathcal{N}_\Phi:=\{u \in X \setminus\{0\};\ \Phi'(u)u=0\},$$
which under some mild conditions on $\Phi$ is a $C^1$ manifold, the so-called {\it Nehari manifold}. While a high amount of works has been applying this procedure (as well as the closely related {\it fibering method} \cite{P}) to a variety of problems in analysis, very few ones have provided an abstract framework showing what is essential to have  $c:=\inf_\mathcal{N} \Phi$ achieved by a ground state solution of \eqref{e}. These abstract results  \cite{BWW,FP, FRQ,I1, SW} are mostly devoted to the situation where $\Phi$ has a local minimum (which is generally assumed to be $\Phi(0)=0$) and displays a uniform mountain-pass geometry, in the sense that for any $u \in X \setminus \{0\}$ the following property holds:\\
\begin{itemize}
	\item [(H1)] the map $t \mapsto \Phi(tu)$, defined in $(0,\infty)$, has a unique critical point $t_u$, which is a global maximum point (in particular $\Phi(t_uu)>0$). \\
\end{itemize} This scenario typically arises in superlinear elliptic problems, as the  prototype 
\begin{equation}
\label{m}
-\Delta u= b(x)|u|^{r-2}u, \quad u \in H_0^1(\Omega),
\end{equation}
where $\Omega \subset \Re^N$ is a bounded domain, $r\in (2,2^*)$, and $b$ is a smooth coefficient such that $b(x)>0$ for all $x \in \Omega$.
This problem corresponds to \eqref{e} with
\begin{equation}\label{f}
\Phi(u)=\frac{1}{2}\int_\Omega |\nabla u|^2 -\frac{1}{r}\int_\Omega b(x)|u|^r, \quad u \in H_0^1(\Omega).
\end{equation}
More generally \cite{SW}, one may consider  the functional
$$\Phi(u)=\frac{1}{2}\|u\|^2 -I(u), \quad u \in H_0^1(\Omega)$$
where  $I(u):=\int_\Omega F(x,u)$, and $F$ is the primitive of a superlinear and subcritical term $f(x,u)$, in which case $\Phi$ is the  energy functional of the problem
\begin{equation}
-\Delta u= f(x,u), \quad u \in H_0^1(\Omega).
\end{equation}
In \cite{SW} the authors dealt with the class of  functionals $\Phi=I_0-I$, where $I_0$ is $p$-homogeneous for some $p>1$ (i.e. $I_0(tu)=t^pI_0(u)$ for any $t>0$ and $u \in X$) and satisfies
$c_0^{-1}\|u\|^p\leq I_0(u)\leq c_0\|u\|^p$
for any $u \in X$ and some $c_0>0$. This functional is inspired by the $p$-Laplacian problem
\begin{equation}
\label{m}
-\Delta_p u=\lambda|u|^{p-2}u+ f(x,u), \quad u \in W_0^{1,p}(\Omega),
\end{equation}
where $\lambda<\lambda_1(p)$, the first eigenvalue of $-\Delta_p$ on $X=W_0^{1,p}(\Omega)$, in which case
$I_0(u)=\int_\Omega \left(|\nabla u|^p - \lambda |u|^p\right)$. The existence of a ground state solution for \eqref{e} was then proved under some `$p$-superlinear' type conditions on $I$, see \cite[Theorem 13]{SW} for the details.
The homogeneity condition on $I_0$ was then removed in \cite[Theorem 1.2]{FRQ}  to handle problems involving nonhomogeneous operators such as the $(p,q)$-Laplacian operator, as well as Kirchhoff and anisotropic type operators.

The main purpose of this article is to set up a framework for functionals that do not have a uniform geometry on $X$, in the sense that the map $t \mapsto \Phi(tu)$ has not the same behavior for all $u \neq 0$. This is the case if
(H1) does not hold in the whole of $X \setminus \{0\}$, but rather in some part of it. A simple example  is given by \eqref{f} with $b$ vanishing in some part of $\Omega$ or changing sign, in which case (H1) holds only for those $u$ such that $\int_\Omega b(x)|u|^r>0$.
So it is reasonable to deal with functionals satisfying (H1) in an open cone of $X$, i.e. an open set $Y \subset X$ such that $t u \in Y$  for every $t>0$ and $u \in Y$. 
Note that  the condition 
$ \int_\Omega b(x)|u|^r>0$
clearly defines an open cone in $X$. 

Furthermore, we shall consider other geometries than (H1). More concretely, we aim at functionals such that the map $t \mapsto \Phi(tu)$ has at least a local minimum point whenever $u$ lies in an open cone.
This condition arises in several problems, as we shall see in our applications. 
Overall we shall provide conditions ensuring that the infimum of $\Phi$ over $\mathcal{N} \cap Y$ is achieved by a critical point of $\Phi$, which can then be considered as a {\it ground state relative to $Y$}, in the sense that it has the least energy among critical points in $Y$. Some further conditions on $\Phi$ over $X\setminus Y$ shall entail that $c$ is the ground state of $\Phi$.

The applications of our abstract results shall concern mostly two classes of functionals. The first one is given by 
\begin{equation}\label{cf}
\Phi=I_0-I
\end{equation}
 where $I_0,I \in C^1(X,\Re)$ are such that $I$ dominates $I_0$ at infinity (i.e $I$ grows faster than $I$, in absolute value) and one the following possibilities occur (in some cases both): 
\begin{itemize}
\item  $I_0$ is coercive in $Y_1=\{u \in X: I(u)>0\}$ so that $(H1)$ holds in $Y_1$, and $\Phi$ has a mountain-pass structure therein.
\item $I$ is anti-coercive in $Y_2=\{u \in X: I_0(u)<0\}$, i.e. $I(u) \to -\infty$ as $\|u\| \to \infty$, with $u \in Y_2$. It follows that $\Phi$ is coercive therein.
\end{itemize}

The first case holds for instance if $I_0(u)\geq C\|u\|^p$  for some $C>0$, $p>1$, and every $u \in Y_1$, and $\frac{I(tu)}{t^p} \to \infty$ as $t \to \infty$, uniformly for $u$ in weakly compact subsets of $Y_1$. This happens in particular if $I\geq 0$ and $Y_1=\{u \in X: I(u)>0\}$, as we shall see in several problems (cf. Corollaries \ref{c2}, \ref{c3} and \ref{ch}).

As for the second case, it only occurs if $I_0$ and $I$ take negative values. This can be observed in the prototype of {\it indefinite} type problems
$$-\Delta u =\lambda u +b(x)|u|^{r-2}u, \quad u \in H_0^1(\Omega),$$
where  $\lambda \in \Re$, $2<r<2^*$, and $b \in L^{\infty}(\Omega)$ changes sign. A similar structure arises in the $(p,q)$-Laplacian problem
 \begin{equation*}
-\Delta_p u -\Delta_q u = \alpha |u|^{p-2}u+\beta |u|^{q-2}u, \quad u \in W_0^{1,p}(\Omega),
\end{equation*}
where $1<q<p$ and $\alpha,\beta \in \Re$. In such case $\Phi=I_0-I$ with
$$I_0(u)=\frac{1}{q}\int_\Omega \left(|\nabla u|^q- \beta|u|^q \right)\quad \mbox{and} \quad I(u)=-\frac{1}{p}\int_\Omega \left(|\nabla u|^p- \alpha|u|^p \right).$$
We consider the open cones $$Y_1=\left\{u \in W_0^{1,p}(\Omega): \int_\Omega \left(|\nabla u|^p- \alpha|u|^p \right)<0\right\}$$ and $$Y_2=\left\{u \in W_0^{1,p}(\Omega): \int_\Omega \left(|\nabla u|^q- \beta|u|^q \right)<0 \right\},$$
which are nonempty if, and only if, $\alpha >\lambda_1(p)$ and $\beta>\lambda_1(q)$, respectively. 

 It turns out that there exist $\alpha_*(\beta) \geq \lambda_1(p)$ and $\beta_*(\alpha)\geq \lambda_1(q)$ having the following properties:
 \begin{enumerate}
\item  For $\beta<\beta_*(\alpha)$ there exists $C>0$ such that $I_0(u)\geq C\|u\|^q \quad \forall u \in \overline{Y}_1$.
\item  For $\alpha<\alpha_*(\beta)$ there exists $C>0$ such that
$ I(u)\le -C\|u\|^p\quad  \forall u \in \overline{Y}_2$. 
 \end{enumerate}
 Since $q<p$, we see that (H1) holds in $Y_1$ for $\beta<\beta_*(\alpha)$, whereas $\Phi$ is coercive in $Y_2$ for $\alpha<\alpha_*(\beta)$. It follows that 
 \begin{enumerate}
 	\item $c_1:=\displaystyle \inf_{\mathcal{N}\cap Y_1}\Phi$, the ground state of $\Phi$ relative to $Y_1$, 	is positive and achieved for $\alpha >\lambda_1(p)$ and $\beta<\beta_*(\alpha)$, 
 	\item $c_2:=\displaystyle \inf_{\mathcal{N}\cap Y_2}\Phi$, the ground state of $\Phi$ relative to $Y_2$, is negative and achieved for $\alpha<\alpha_*(\beta)$ and $\beta>\lambda_1(q)$. 
 \end{enumerate}
 In addition, it is clear that $\Phi \geq 0$ on $\mathcal{N} \setminus Y_2$, so that $c_2$ is the ground state level whenever it is achieved, whereas $c_1$ is the ground state level for $\alpha >\lambda_1(p)$ and $\beta< \lambda_1(q)$, since in this case $Y_2$ is empty and $\mathcal{N} \subset Y_1$.
This kind of structure also arises in a Kirchhoff equation (see section \ref{fe}), as well as in a fourth-order problem associated to a strongly coupled system (see section \ref{sh}).

 The second class of functionals to be considered has the form \begin{equation}
\label{fcc1}
\Phi=\frac{1}{\kappa}P +\frac{1}{\upsilon}\Upsilon+\frac{1}{\gamma}\Gamma
\end{equation}
where $P,\Upsilon,\Gamma \in C^1(X,\Re)$ are weakly lower semicontinuous and $\kappa$-homogeneous, $\upsilon$-homogeneous and $\gamma$-homogeneous, respectively, for some $0<\upsilon<\kappa<\gamma$. In particular, the sets 
$$Y_1=\left\{u \in X: \Upsilon(u)<0\right\} \quad
\mbox{and}
\quad Y_2=\left\{u \in X: \Gamma(u)<0\right\},$$
are open cones and we assume that for any $u \in Y=Y_1 \cap Y_2$ the map $t \mapsto \Phi(tu)$ has exactly two critical points (both nondegenerate) $0<t_u<s_u$, which are a local minimum and a local maximum points, respectively. Under these conditions, $c_1:=\displaystyle \inf_{\mathcal{N}\cap Y_1}\Phi$ is the ground state level of $\Phi$, cf. Corollary \ref{cc} below. We shall consider several functionals having this structure,  among which we highlight the one given by
$$P(u,v)=\|v\|_q^q, \quad \Upsilon(u,v)=\int_\Omega \left(|\nabla u|^2+|\nabla v|^2-2\lambda uv\right), \quad \mbox{and} \quad \Gamma(u,v)=-\|u\|_r^r,$$
defined in $X:=H_0^1(\Omega)\times H_0^1(\Omega)$, 
where $\lambda \in \Re$ and $2<q<r<2^*$. In this case $\Phi$ is the energy functional associated to the semilinear gradient system
\begin{equation*}
\left\{
\begin{array}{ll}
-\Delta u = \lambda v+|u|^{r-2}u \
& \mbox{in} \ \ \Omega, \ \ \\
-\Delta v =  \lambda u-|v|^{q-2}v \
&\mbox{in} \ \ \Omega, \ \ \\
u = v = 0 \ &\mbox{on} \ \
\partial\Omega,
\end{array}
\right.
\end{equation*}
which is treated with more generality in section \ref{kayeexample1}.

We stress that our abstract setting is new not only for homogeneous operators as the $p$-Laplacian, but also for nonhomogeneous ones as the $(p,q)$-Laplacian, Kirchhoff type operators, etc. For simplicity, we apply our results to elliptic problems with Dirichlet boundary conditions, but other boundary conditions can be handled, as well as problems in unbounded domains. Finally, let us mention that for the sake of space we have limited the bibliography to works dealing with the class of problems considered here via the Nehari manifold method (or the fibering method). 

The outline of this article is the following: in section 2 we state and prove our main abstract results. These ones are organized according to the behavior of the map $t \mapsto \Phi(tu)$ for $u \in Y$, where $Y$ is a nonempty open cone. In section 3 we derive some results for the classes of functionals \eqref{cf} and \eqref{fcc1} and discuss some examples. Finally, some applications of our results to pdes and systems of pdes are given in sections 4 and 5, respectively.

\medskip

\subsection*{Notation} Throughout this article, we use the following notation:

\begin{itemize}
	\item Unless otherwise stated $\Omega$ denotes a bounded domain of $\Re^N$ with $N\geq 1$.
		
	\item Given $r>1$, we denote by $\Vert\cdot\Vert_{r}$ (or $\Vert\cdot\Vert_{r,\Omega}$ in case we need to stress the dependence on $\Omega$) the usual norm in
	$L^{r}(\Omega)$, and by $r^*$ the critical Sobolev exponent, i.e. $r^*=\frac{Nr}{N-r}$ if $r<N$ and $r^*=\infty$ if $r \geq N$.
	
	\item Strong and weak convergences are denoted by $\rightarrow$ and
	$\rightharpoonup$, respectively.
	
	\item Given $f\in L^{1}(\Omega)$, we set $f^{\pm}:=\max(\pm f,0)$. The integral $\int_{\Omega}f$ is considered
	with respect to the Lebesgue measure. Equalities and inequalities involving
	$f$ shall be understood holding \textit{a.e.}.
	
	\item $\varphi_u:[0,\infty) \to \Re$ is the {\it fibering map} given by $\varphi_u(t)=\Phi(tu)$ for any $u \in X$.
	
	\item $J:X \rightarrow \Re$ is given by $J(u)=\Phi'(u)u$, and $\mathcal{N}:=J^{-1}(0) \setminus \{0\}$.
	
	\item $S$ is the unit sphere in $X$.
	
	\item $\partial Y$ and $\overline{Y}$ denote the boundary and the closure of $Y \subset X$ in the weak topology, respectively.
	
	\item Given $u \in X$ and $R>0$ we denote by $B(u,R)$ the open ball of radius $R$ around $u$.
	
	\item Unless otherwise stated, $C>0$ is a constant whose value may vary from one ocurrence to the other.
\end{itemize}

\medskip
\section{Main results}
\medskip

In this section $Y \subset X \setminus \{0\}$ is assumed to be a nonempty open cone, where $X$ is a reflexive Banach space, and $\Phi \in C^1(X,\Re)$ with $\Phi(0)=0$. We also assume that $\Phi$ and $J$ are weakly lower semicontinuous on $X$.

We shall consider different situations in accordance with the geometry of the fibering maps $\varphi_u$, for $u\in Y$. Overall we aim at finding conditions on $\Phi$ so that $$c:=\displaystyle \inf_{\mathcal{N}\cap Y}\Phi$$ is achieved and provides the ground state level of $\Phi$. 
A basic and general property needed in our minimization procedure is
the following  compactness condition:\\
\begin{itemize}
	\item [(HY)$_d$] If $(u_n) \subset \mathcal{N}\cap Y$ and $\Phi(u_n)\to d \in \mathbb{R}$, then $(u_n)$ has a subsequence  weakly convergent in $Y$. \\
\end{itemize} 

Our setting shall include two main situations in accordance with the number of critical points of the map $\varphi_u$ for $u \in Y$. First we assume that $\varphi_u$ has a unique critical point. In this case we do not require $J$ to be $C^1$ (although in most applications it is), so $\mathcal{N}$ is not necessarily a manifold.
\medskip
\subsection{Cones where $\varphi_u$ has a unique critical point}\strut\\

To begin with, we assume that (H1) holds within $Y$:

\begin{theorem}\label{thm2}
	Assume $(HY)_c$ and (H1) for every $u \in Y$. Then $c>0$  is  achieved by a critical point of $\Phi$. If, in addition,
	$J(u)\neq 0$ for every $u \in X \setminus Y$ with $u \neq 0$ then $c$ is the ground state level of $\Phi$.
\end{theorem}

\begin{proof}
	We split the proof into two parts. First we proceed as in \cite{FRQ,SW} to show that $c$ is achieved. Given $u \in Y$ we know by (H1) that there exists exactly one $t_u>0$ such that $t_u u\in \mathcal{N} \cap Y$. In particular, this set is nonempty. It is also clear that $c \geq 0$. 
	Let $(u_n) \subset \mathcal{N} \cap Y$ be
	a minimizing sequence for $c$, i.e. $\Phi(u_n) \to c$. By $(HY)_c$ we can assume that $u_n \rightharpoonup u\in Y $. Since $\Phi$ is weakly lower semicontinuous and $t_{u_n}=1$  for every $n$, we have
	$c\leq \Phi(t_u u)\leq \liminf \Phi(t_u u_n) \leq \lim \Phi(u_n)=c$,
	i.e.  $\Phi(t_u u)=c>0$.  
	Finally, it is clear that $\mathcal{N} \subset Y$ if $J(u)\neq 0$ for every $u \in X \setminus Y$ with $u \neq 0$, so that $c=\inf_{\mathcal{N}}\Phi$ in this case.
	
	Next we proceed as in \cite{LWW} (see also \cite{CW}) to show that $c$ is a critical value of $\Phi$.
	Assume that $c$ is achieved by $u_0$ and this one is not a critical point of $\Phi$. Then there exists  $v \in X$ such that 
	$\Phi'(u_0)v <0$.
	By continuity we can find $\varepsilon, \delta >0$ small such that 
	\begin{equation}\label{bbs}
	\Phi'(t(u_0+ \sigma v))v <0,\;\;\; \mbox{ for } \ \  t \in [1-\epsilon, 1+ \epsilon] \ \ \mbox{and } \ \   \sigma  \in [-\delta, \delta].
	\end{equation}
	Moreover, since $$\Phi'((1-\epsilon)u_0)u_0=\varphi_{u_0}'(1-\epsilon)>0>\varphi_{u_0}'(1+\epsilon)=\Phi'((1+\epsilon)u_0)u_0$$ and $Y$ is open, there exists $\overline{\sigma} \in (0,\delta)$ such that $u_0+\overline{\sigma}v \in Y$ and $\varphi_{u_0+\overline{\sigma}v}'(1-\epsilon)>0>\varphi_{u_0+\overline{\sigma}v}'(1+\epsilon)$.
	It follows that $	t_{u_0+\overline{\sigma}v} \in (1-\epsilon, 1+\epsilon)$. Writing $\overline{t}={t}_{u_0+\overline{\sigma}v}$, 
	from (\ref{bbs}) we have
	$$
	\Phi(\overline{t}(u_0+ \overline{\sigma} v)) - \Phi(u_0)\leq   \Phi(\overline{t}(u_0+ \overline{\sigma} v)) - \Phi(\overline{t}u_0)=  \overline{t}\int^{\overline{\sigma}}_{0} \Phi'(\overline{t}(u_0+ \sigma v))v d\sigma <0,
	$$
	so that
	$
	\Phi(\overline{t}(u_0+ \overline{\sigma} \phi)) < \Phi(u_0) = c$,
	which is a contradiction. Therefore $\Phi'(u_0)=0$ and the proof is complete.\\
\end{proof}

Next we provide some conditions leading to $(HY)_d$ when $\Phi$ has a mountain-pass geometry in $Y$:

\begin{lem}\label{l1}
	Assume (H1) for every $u \in Y$, and the following conditions: 
	\begin{enumerate}
		\item There exists $\sigma>1$ such that $\displaystyle \lim_{t \to \infty} \frac{\Phi(tu)}{t^\sigma}=-\infty$ uniformly for $u$ on weakly compact subsets of $Y$.
		\item $\Phi\ge I_0-I$ in $Y$,
		where $I$ is weakly continuous and vanishes on $ \partial Y$, and $I_0$ satisfies $\dis \lim_{t \to \infty} I_0(tu)=\infty$ uniformly for $u \in \mathcal{S} \cap Y$.
	\end{enumerate}	
	If $(u_n) \subset \mathcal{N} \cap Y$ and $(\Phi(u_n))$ is bounded then $(u_n)$ is bounded. In particular, $(HY)_d$ holds for any $d \in \Re$ if we assume in addition that
	\begin{enumerate}
		\item [(3)] $J>0$ on $\partial Y \setminus \{0\}$ and in
		$Y \cap B(0,R)$, for some $R>0$. 
	\end{enumerate}
	
\end{lem}

\begin{proof}
	If $(u_n) \subset \mathcal{N} \cap Y$ is unbounded then we may assume that $\|u_n\| \to \infty$ and $v_n \rightharpoonup v$ in $X$, where $v_n=\frac{u_n}{\|u_n\|}$. If $v \in \partial Y$ then, since $t_{u_n}=1$ for every $n$ and $I$ is weakly continuous and vanishes on $ \partial Y$, for every $t>0$ we have
	\begin{equation}
	\label{b1}\Phi(u_n)\geq \Phi(tv_n)\geq I_0(tv_n)-I(tv_n)\ge I_0(tv_n) - C.
	\end{equation}
	By (2) we have $I_0(tv_n) \to \infty$ uniformly as $t \to \infty$, and
	we obtain a contradiction with the boundedness of $\Phi(u_n)$. Hence $v  \in Y$ and consequently
	$$\frac{\Phi(u_n)}{\|u_n\|^{\sigma}}= \frac{\Phi(\|u_n\|v_n)}{\|u_n\|^{\sigma}} \rightarrow -\infty,$$
	which contradicts the fact that $\Phi(u_n)>0$ for every $n$.
	Therefore $(u_n)$ is bounded. Let us assume in addition (3). Up to a subsequence, we have $u_n \rightharpoonup u$ in $X$. Since $(u_n) \subset Y$ we have either $u \in Y$ or $u\in \partial Y$. From $J>0$ in $Y \cap B(0,R)$ we know that $(u_n)$ is bounded away from zero, so that repeating \eqref{b1} we infer that $u \neq 0$. The weak lower semicontinuity of $J$ yields
	$J(u) \leq \liminf J(u_n)=0$, and since $J>0$ on $\partial Y \setminus \{0\}$ we deduce that $u \in Y$, so that $(HY)_d$ is satisfied for any $d \in \Re$.
\end{proof}


The boundedness of sequences minimizing $\Phi$ over $\mathcal{N} \cap Y$ can sometimes be obtained by other conditions (e.g. the boundedness of $\mathcal{N} \cap Y$ or the coercivity of $\Phi$ over it). In such cases $(HY)_d$ might be easier to establish through the following condition:\\

\begin{itemize}
	\item [(HJ)]  If $(u_n)\subset Y$, $u_n\rightharpoonup u$ in $X$ and  $J(u_n)\to J(u)$ then $u_n \to u$ in $X$.\\
\end{itemize}

Let us note that (HJ) is satisfied if $J(u)=C\|u\|^{\theta}+H(u)$ where $C,\theta>0$ and $H$ is weakly continuous (which is the case in most of our applications).
\begin{lem}\label{l2}
$(HY)_d$ is satisfied under the following conditions:
	\begin{enumerate}
		\item Any sequence $(u_n) \subset \mathcal{N}\cap Y$ satisfying $\Phi(u_n)\to d$ is bounded.
		\item  $J>0$ on $\partial Y \setminus \{0\}$ and in
		$Y \cap B(0,R)$, for some $R>0$. 
		\item $(HJ)$ holds for $u=0$.
	\end{enumerate}
\end{lem}

\begin{proof}
	By our assumptions, we may assume that $u_n \rightharpoonup u$ in $X$ if $(u_n) \subset \mathcal{N}\cap Y$ and $\Phi(u_n)\to d$. Since $J$ is weakly lower semicontinuous, we have $J(u)\leq 0$. Note also that (2) implies that $\|u_n\| \geq R$ for every $n$. If $u \not \in Y$ then $u \in \partial Y$, and (2) yields that $u=0$, so that $J(u_n) \to J(u)$. By (3) we have $u_n \to 0$ in $X$, which is a contradiction.
\end{proof}

Next we consider the following behavior on $Y$:\\

\begin{itemize}
	\item [(H2)] $\varphi_u$ has a unique critical point $t_u>0$, which is a global minimizer (and consequently $\Phi(t_uu)=\varphi_u(t_u)<0$) \\
\end{itemize}

\begin{theorem}
	\label{tp1}
	Assume  (H2) for any $u \in Y$, and the following conditions:
	\begin{enumerate}
		\item  $\Phi$ is coercive in $Y$, i.e. $\Phi(u_n) \to \infty$ if $(u_n) \subset Y$ and $\|u_n\| \to \infty$.
		\item $\Phi \geq 0$ on $\partial Y$.
	\end{enumerate} 
	Then $c<0$ and $c=\inf_{Y} \Phi$, so that $c$ is a local minimum of $\Phi$. If, in addition, $\Phi \geq 0$ on $\mathcal{N} \setminus Y$, then $c$ is the ground state level of $\Phi$.
	
\end{theorem}

\begin{proof}
	By (1) we know that $\Phi$ is bounded from below on $Y$. Let $(u_n) \subset Y$ be such that
	$\Phi(u_n) \to \inf_{\overline{Y}} \Phi=\inf_Y \Phi<0$. From (1), (2) and the weak lower semicontinuity of $\Phi$ we can assume that $u_n \rightharpoonup u \in Y$. So $\Phi(u)=\inf_{Y} \Phi$ and since $Y$ is open, this infimum is a local minimum (and therefore a critical value) of $\Phi$. It follows that $c=\inf_{Y} \Phi$ and  $c$ is the ground state level of $\Phi$ if $\Phi \geq 0$ on $\mathcal{N} \setminus Y$.\\
\end{proof}

\begin{rmk}
From the previous proof it is clear that instead of assuming that $\Phi$ is coercive in $Y$ and satisfies (H2) for any $u \in Y$, it suffices to assume $(HY)_c$ and $\inf_Y \Phi<0$. Furthermore, $Y$ may be an open set in general.\\
\end{rmk}

\subsection{Cones where $\varphi_u$ has more than one critical point}\strut\\

We allow now $\varphi_u$ to have more than one critical point. In contrast with the previous susbsection, we shall require more regularity on $\varphi_u$ for $u \in Y$. First we assume the following condition in $Y$:\\

\begin{itemize}
	\item [(H3)] $\varphi_u \in C^2(0,\infty)$ has a non-degenerate minimum point $t_u>0$ such that $\varphi_u'<0$ in $(0,t_u)$. Moreover $\varphi_u(t_u)<\varphi_u(t)$ for any $t>0$ such that $\varphi_u'(t)=0$. In particular $\Phi(t_u u)=\varphi_u(t_u)<0$. \\
\end{itemize}

It is clear that $t_u$ in (H3) is the first nonzero critical point of $\varphi_u$. (H3) holds in particular if $\varphi_u$ has two critical points, the first one being a local minimum point. This condition arises in several problems, as we shall see in the next sections.
\begin{theorem}\label{tp2} Assume $(HY)_c$ and (H3), (HJ) for all $u\in Y$. 
	Then  $c<0$ and it is achieved.
	In addition:
	\begin{enumerate}
		\item  If $J$ is $C^1$ in $Y$ then $c$ is a local minimum of $\Phi$.  	
		\item  If $\Phi \ge 0$ on $\mathcal{N}\setminus Y$, then  $c$ is the ground state level of $\Phi$.
	\end{enumerate}

\end{theorem}
\begin{proof} By (H3) we know that $\mathcal{N}\cap Y$ is nonempty and $c<0$.
	Let $(u_n)\subset \mathcal{N}\cap Y$ satisfy $\Phi(u_n)\to c$. By $(HY)_c$ we may assume that $u_n\rightharpoonup u\in Y$, so there exists $t_u>0$ such that $J(t_uu)=0$. We claim that $u_n\to u$ in $X$. Indeed, on the contrary, we conclude from (HJ) that $0=J(t_uu)<\limsup J(t_uu_n)$,
	so that for $n$ large enough we may assume that $J(t_uu_n)>0$ i.e. $\varphi'_{u_n}(t_u)>0$. Thus (H3) yields that $t_u>t_{u_n}=1$, which implies that $\Phi(t_uu)<\Phi(u)\le \liminf \Phi(u_n)=c$,
	a contradiction. Hence $u_n\to u$ in $X$ and $\varphi_u'(1)=0$.  If $t_u \neq 1$ then from (H3) we deduce that $\Phi(t_uu)<\Phi(u)=c$, which is a contradiction. Therefore $t_u=1$ and $\Phi(u)=c<0$. 
	
	From now on we assume that $u$ is the minimizer of the last paragraph and $J$ is $C^1$ over $Y$. Let $F:Y\times (0,\infty)\to \mathbb{R}$ be given by $F(v,t):=\varphi_v'(t)=J(tv)/t$. It follows that $F\in C^1$, $F(u,1)=0$ and $F_t(u,1)>0$. By the implicit function theorem, there exists an open ball $B$ containing $u$ and a unique $C^1$ map $\sigma: B \to (0,\infty)$ such that $\sigma(v)=t_v$ and $F(v,\sigma(v))=0$ for any $v \in B$. Moreover, since $F_t(u,1)>0$ and $F_t$ is continuous, making $B$ smaller if necessary, we have $F_t(v,1)>0$ for $v \in B$.

	We claim that there exists $R>0$ such that $t_ww\in B$ for any $w\in B(u,R)$. Otherwise, there exists a sequence $R_n\to 0^+$ and $w_n\in B(u,R_n)$ such that $t_{w_n}w_n\notin B$. Since $w_n\to u$, by the previous paragraph we deduce that $t_{w_n}\to 1$ and hence $t_{w_n}w_n\to u$, a contradiction. Making $R>0$ smaller if necessary, we have $B(u,R) \subset B$.
	It follows that for any $w\in B(u,R)$ the line segment between $w$ and $t_w w$ lies inside $B$.
	Thus, if $t_w\neq 1$ and $w \in B(u,R)$ then there exists $\theta\in(\min\{1,t_w\},\max\{1,t_w\})$ such that
	\begin{equation*}	 
	\varphi_w(1)-\varphi_{w}(t_w)=\varphi_w'(t_w)(t_w-1)+\frac{1}{2}\varphi_w''(\theta)(t_w-1)^2=\frac{1}{2}F_t(w,\theta)(t_w-1)^2,
	\end{equation*}
	and from $F_t(w,\theta)=\theta^{-2}F_t(\theta w,1)>0$, we conclude that $\Phi(t_ww)\le\Phi(w)$, so that
	$
	\Phi(u)\le \Phi(t_{w}w)\le  \Phi(w)$.
	Since this inequality holds for each $w\in B(u,R)$
	we see that $u$ is a local minimizer of $\Phi$, and consequently a critical point of $\Phi$. Finally, (2) follows from $c<0$. 
\end{proof}

\begin{rmk}\label{N0empty}
	{\rm It is easily seen that Theorem \ref{tp2} still holds if the condition $\Phi\ge 0$ on $ \mathcal{N}\setminus Y$ is relaxed to $c<\displaystyle \inf_{\mathcal{N}\setminus Y}\Phi$.
	}
\end{rmk}

Lastly, we show that the following condition leads to the existence of a critical point (generally not a ground state):\\

\begin{itemize}
	\item [(H4)] $\varphi_u \in C^2(0,\infty)$ has a non-degenerate maximum point $s_u>0$ such that $\varphi_u'<0$ in $(s_u,\infty)$ and $\varphi_u(s_u)\geq \varphi_u(t)$ for any $t>0$ such that $\varphi_u'(t)=0$. \\
\end{itemize}

Note that $s_u$ in (H4) is the last critical point of $\varphi_u$ and satisfies $\varphi_u(s_u)\geq \varphi_u(t)$ whenever $\varphi_u'(t)>0$.
We stress that it might happen that $\Phi(s_u u)=\varphi_u(s_u)<0$. In particular, $s_u$ does not need to be a global maximum point of $\varphi_u$.

\begin{theorem}\label{tp3} Assume (H4),(HJ) for all $u\in Y$ and $(HY)_d$ with $d:=\inf\{\Phi (s_uu):\ u\in Y\} $. If $J$ is $C^1$ in $Y$ then $d$ is achieved by a critical point of $\Phi$.
\end{theorem}
\begin{proof} By (H4) we know that $\{s_uu:\ u\in Y\}$ is nonempty.
	Let $(u_n)\subset Y$ satisfy $\Phi(s_{u_n}u_n)\to d$ and set $v_n:=s_{u_n}u_n$. By $(HY)_d$ we may assume that  $v_n \rightharpoonup v\in Y$, so there exists $s_v>0$ such that $J(s_vv)=0$. We claim that $s_v v$ achieves $d$. 
	
	Let us first assume that $v_n \not \to v$ in $X$.  From (HJ) we deduce that $0=J(s_vv)<\limsup J(s_vv_n)$,
	so that for $n$ large  we must have  $J(s_vv_n)>0$ i.e. $\varphi'_{v_n}(s_v)>0$. By (H4) we must have $s_v<1$, which implies that $\Phi(s_vv)\leq \liminf \Phi(s_vv_n) \le \liminf\Phi(v_n)=d$, i.e. $\Phi(s_v v)=d$.
	
	Assume now that $v_n\to v$ in $X$, so that $v \in \mathcal{N}$. 
	Let us prove that $s_v=1$, which implies $\Phi(v)=d$.  To this end, note that $F(v,s_v)=\varphi_v'(s_v)=0$ and $F_t(v,s_v)=\varphi_v''(s_v)<0$, where $F$ is as in the proof of Theorem \ref{tp2}. By the implicit function theorem, there exists an open ball $B$ containing $v$ and a unique $C^1$ map $\sigma: B \to (0,\infty)$ such that $\sigma(u)=s_u$ and $F(u,\sigma(u))=0$ for any $u \in B$. Since $v_n\to v$ we have $s_v=\sigma(v)=\lim \sigma(v_n)=\lim s_{v_n}=1$, and the claim is proved.
	
	Furthermore, the previous discussion also shows that the set $\{s_uu:\ u\in Y\}=\{u \in Y: J(u)=0,\ J'(u)u<0\}$ is, in fact, a $C^1$ manifold. Since $J'(u)$ is surjective, we conclude from the Lagrange multipliers rule that $ \Phi'(v)=\alpha J'(v)$ for some $\alpha\in \mathbb{R}$.
	From $\Phi'(v)v=0>J'(v)v$, it follows that $\alpha=0$ and the proof is complete.
\end{proof}


\medskip

\section{Some classes of functionals and examples}

\medskip

Let us apply the results of the previous section to two classes of functionals. Recall that $X$ is a reflexive Banach space, $Y \subset X \setminus \{0\}$ is a nonempty open cone, and $c:=\displaystyle \inf_{\mathcal{N}\cap Y}\Phi$.
First we deal with $$\Phi=I_0-I,$$ where $I_0,I \in C^1(X,\Re)$  and $I_0(0)=I(0)=0$. The following result, which is a consequence of Theorem \ref{thm2} and Lemma \ref{l1}, extends \cite[Theorem 13]{SW} and \cite[Theorem 1.2]{FRQ}, as far as the existence of a ground state is concerned. Let us stress that $I_0$ may be nonhomogeneous (and this will be the case in several applications). 
\begin{cor}\label{c1}
Under the above conditions, assume in addition that:
	\begin{enumerate}
		\item $I_0$ and $u \mapsto I_0'(u)u$ are weakly lower semicontinuous and $I'$ is completely continuous, i.e. $I'(u_n) \to I'(u)$ in $X^*$ if $u_n \rightharpoonup u$ in $X$.
		\item There exist $C>0$ and $\eta>1$ such that
		$I_0'(u)u\geq C\|u\|^\eta$ for every $u \in \overline{Y}$
		and $I'(u)=o(\|u\|^{\eta-1})$ as $u \to 0$ in $\overline{Y}$.
		\item  $I(u)=I'(u)u=0$ for every $u \in \partial Y$.
		\item There exists $\sigma>1$ such that $t \mapsto \frac{I_0'(tu)u}{t^{\sigma-1}}$ and $t \mapsto \frac{I'(tu)u}{t^{\sigma-1}}$ are nonincreasing and increasing in $(0,\infty)$, respect., for every $u \in Y$. Moreover, $\displaystyle \lim_{t \to \infty} \frac{I_0(tu)}{t^{\sigma}}<\infty=\displaystyle \lim_{t \to \infty} \frac{I(tu)}{t^{\sigma}}$ uniformly for $u$ on weakly compact subsets of $Y$.
	\end{enumerate}
	Then $c$  is positive and achieved by a critical point of $\Phi$. If, in addition,
	$I'(u)u \leq 0<I_0'(u)u$ for $u \in X \setminus Y$ with $u\neq 0$  then  $c$ is the ground state level of $\Phi$.
\end{cor}

\begin{proof}
	First note that since $I'$ is completely continuous, the maps $I$ and $u \mapsto I'(u)u$ are weakly continuous, so that by (1)  we have that $\Phi$ and $J$ are weakly lower semicontinuous.
	We also see that (2) implies that $I_0(u)\geq Cr^{-1}\|u\|^{\eta}$ for any $u \in \overline{Y}$ and $I(u)=o(\|u\|^{\eta})$ as $u \to 0$ in $\overline{Y}$. It follows that for any $u\in Y$ we have $\varphi_u$ positive near zero. By (4) we have $\displaystyle \lim_{t \to \infty} \frac{\Phi(tu)}{t^\sigma}=-\infty$ uniformly for $u$ on weakly compact subsets of $Y$, and $t \mapsto \Phi'(tu)u$ vanishes at most once, so (H1) holds for $u \in Y$. Note also that (2) and (3) yield that $J>0$ on $\partial Y \setminus \{0\}$ and in
	$Y \cap B(0,R)$, for some $R>0$.
   From Lemma \ref{l1} we infer that $(HY)_c$ holds. Finally, if $I'(u)u \leq 0$ for $u \in X \setminus Y$ then $J(u)=I_0'(u)u-I'(u)u>0$ for every $u \in X \setminus Y$ with $u \neq 0$. Theorem \ref{thm2} yields then the conclusion.
\end{proof}

\begin{rmk}\label{rc1}
Instead of $I'(u)=o(\|u\|^{\eta-1})$ as $u \to 0$ in $\overline{Y}$, one may require more generally that for some $C,R>0$ we have $J(u)\geq C\|u\|^{\eta}$  for $u \in Y \cap B(0,R)$. Indeed, in such case $\varphi_u$ is still positive near zero and the rest of the proof remains valid. See Corollaries \ref{c2} and \ref{c3} for  examples where this condition occurs. 
\end{rmk}

Let us consider now the functional
\begin{equation}
\label{fcc}
\Phi(u):=\frac{1}{\kappa}P(u) +\frac{1}{\upsilon}\Upsilon(u)+\frac{1}{\gamma}\Gamma(u)
\end{equation}
defined in $X$.
We assume that  $0<\upsilon<\kappa<\gamma$ and $P,\Upsilon,\Gamma \in C^1(X,\Re)$ are weakly lower semicontinuous and $\kappa$-homogeneous, $\upsilon$-homogeneous and $\gamma$-homogeneous, respectively. Moreover there exists $C>0$ such that $ P(u)\geq C^{-1}\|u\|^{\kappa}$, $|\Upsilon(u)| \leq C\|u\|^\upsilon$ and $|\Gamma(u)|\leq C\|u\|^{\gamma}$ for all $u \in X$.

Note that by homogeneity we have
$J=P+\Upsilon+\Gamma$,
which is assumed to satisfy condition (HJ).
We deal with the cones $$Y_1=\left\{u \in X: \Upsilon(u)<0\right\} \quad
\mbox{and}
\quad Y_2=\left\{u \in X: \Gamma(u)<0\right\},$$
under the following condition:\\

\begin{enumerate}
	\item [(H5)] For any $u \in Y_1 \cap Y_2$ the map $\varphi_u$ has exactly two critical points $0<t_u<s_u$, both non-degenerate, with $t_u$ a local minimum point, and $s_u$ a local maximum point.\\
\end{enumerate}

This condition enables us to apply Theorems \ref{tp2} and \ref{tp3} on $Y_1$ and $Y_2$, respectively, and derive the following result:

\begin{cor}\label{cc}
Under the above conditions, assume that $Y_1$ is non-empty. Then $c:=\inf_{\mathcal{N}\cap Y_1} \Phi<0$ is the ground state level, and it is achieved by a local minimizer of $\Phi$. If, in addition, $Y_2$ is nonempty then $\Phi$ has a second critical point.
\end{cor}

\begin{proof}
	First of all, note that by (3) we have
	\begin{equation}\label{coercibe}
	\Phi(u)=\frac{\gamma-\kappa}{\kappa\gamma}P(u)+\frac{\gamma-\upsilon}{\upsilon\gamma}\Upsilon(u) 
	\ge C_1\|u\|^\kappa-C_2\|u\|^\upsilon \quad \forall u \in \mathcal{N},
	\end{equation}
	where $C_1,C_2$ are positive constants. Therefore $\Phi$ is coercive over $\mathcal{N}$.
	
	Let us
	show that Theorem \ref{tp2} provides the desired assertions on $c$.  We analyze $\varphi_u$ for  $u \in Y_1$. If $\Gamma(u)\ge 0$, then
	\begin{enumerate}
		\item[i)] $\varphi_u$ has a unique critical point $t_u>0$, which is a non-degenerate global minimizer, 
	\end{enumerate}
	whereas if $\Gamma(u)< 0$, then $u \in Y_1 \cap Y_2$, and by (H5) 
	\begin{enumerate}
		\item[ii)] $\varphi_u$ has exactly two  critical points (both non-degenerate), the first one being a local minimum point $t_u>0$.
	\end{enumerate}
	Moreover, in both cases $\varphi_u(t_u)=\Phi(t_u u)<0$, and we deduce that $(H3)$ holds for $u \in Y_1$. To prove $(HY)_{c}$ note that, by \eqref{coercibe}  we can find a sequence $(u_n) \subset \mathcal{N}\cap Y_1$ such that $u_n\rightharpoonup u$ in $X$ and $\Phi(u_n) \to c$. Since $c<0$, we conclude that $u\neq 0$, and from the inequality $-(\gamma-1)\varphi_{u_n}'(1)+\varphi_{u_n}''(1)>0$ it follows that 
	$$P(u)+\frac{\gamma-\upsilon}{\gamma-\kappa}\Upsilon(u)\leq \liminf \left( P(u_n)+\frac{\gamma-\upsilon}{\gamma-\kappa}\Upsilon(u_n)\right)\leq 0,$$
	Since $u \neq 0$ we infer that $P(u)>0$ and thus
	$\Upsilon(u)<0$ i.e. $u\in Y_1$. Thus $(HY)_c$ holds.
	Finally, $J$ is clearly $C^1$ and by \eqref{coercibe} we have $\Phi>0$ in $\mathcal{N} \setminus Y_1$.
	
	Next we show that Theorem \ref{tp3} applies on $Y_2$. Given $u \in Y_2$ we have the following alternatives: if $\Upsilon(u)\ge 0$, then
	\begin{enumerate}
		\item[i)] $\varphi_u$ has only one critical point $s_u>0$, which is a non-degenerate global maximum point. 
	\end{enumerate}
	On the other hand, if $\Upsilon(u)< 0$, then we have again $u \in Y_1 \cap Y_2$, so that by (H5) 
	\begin{enumerate}
		\item[ii)] $\varphi_u$ has only two  critical points (both non-degenerate), the second one being a local maximum point $s_u$.
	\end{enumerate}
	Thus (H4) holds for any $u \in Y_2$. We set $d:=\inf\{\Phi (s_uu):\ u\in Y_2\}$ and prove $(HY)_{d}$. Once again from \eqref{coercibe}, we see that if $(u_n) \subset \mathcal{N}\cap Y_2$ and $\Phi(s_{u_n}u_n) \to d$ then we can assume that $v_n:=s_{u_n}u_n\rightharpoonup v$ in $X$. From $-(\upsilon-1)\varphi_{v_n}'(1)+\varphi_{v_n}''(1)<0$, it follows that 
	$$C^{-1}\|v_n\|^\kappa\leq P(v_n)<-\frac{\gamma-\upsilon}{\kappa-\upsilon}\Gamma(v_n)\leq C\|v_n\|^{\gamma}$$
	Hence there exists $C>0$ such that $\|v_n\|\ge C$, and consequently $\Gamma(v)<0$, i.e. $v\in Y_2$.
	Therefore $d$ is achieved by a critical point of $\Phi$.
\end{proof}

\begin{rmk}\label{rcc}\strut
\begin{enumerate}
\item An abstract result similar to Corollary \ref{cc} has been proved in \cite{BW}. We point out that in  \cite{BW} the functionals $\Upsilon$ and $\Gamma$ are assumed to be weakly continuous.\\

\item From the previous proof we see that Corollary \ref{cc} still holds if the inequality
$ P(u)\geq C^{-1}\|u\|^{\kappa}$ holds for $u \in \mathcal{N}$. More precisely, it must hold along minimizing sequences for $c$ and $d$. 
\end{enumerate}	
\end{rmk}

Several problems are associated to energy functionals having the structure above. 
Let $a,b \in L^{\infty}(\Omega)$ with $a^+,b^+ \not \equiv 0$, $\lambda>0$, and $1<q<p<r<p^*$. The functional given by \eqref{fcc} with  
\begin{equation}\label{ecc}
P(u)=\|u\|^p, \quad \Upsilon(u)=\Upsilon_\lambda(u)= -\lambda \int_\Omega a|u|^{q}  \quad \mbox{ and } \quad \Gamma(u)=-\int_\Omega b|u|^{r}, 
\end{equation} 
(or $\Upsilon(u)=-\lambda \int_\Omega au$) defined in $X=W_0^{1,p}(\Omega)$, 
has been treated by many authors (see e.g. \cite{BW,I,Ta}).
It is easily seen that such $P$, $\Upsilon,\Gamma$ satisfy the conditions of Corollary \ref{cc}.
Note also that for any $\lambda>0$  we have $Y_1=\left\{u \in X: \int_\Omega a|u|^q>0\right\}$. 
Set
\begin{equation*}\label{ls}
\lambda^*:=  \inf_{u\in Y_1 \cap Y_2} \lambda(u), \quad \mbox{where} \quad \lambda(u):=\frac{r-p}{r-q}\left(\frac{p-q}{r-q}\right)^{\frac{p-q}{r-p}}\frac{\|u\|^{p\frac{r-q}{r-p}}}{(\int_\Omega a|u|^q) \left(\int_\Omega b|u|^r\right)^{\frac{p-q}{r-p}}}
\end{equation*}
is the so-called nonlinear Rayleigh's quotient (see \cite{I1}).  We observe that $\lambda(u)$ is $0$-homogeneous and it  is obtained as the unique solution with respect to $(t,\lambda)$ of the system $\varphi_u'(t)=\varphi_u''(t)=0$.  One can easily show that $\lambda^*$ is achieved and consequently $\lambda^*>0$. It follows that (H5) holds for $0<\lambda<\lambda^*$, and Corollary \ref{cc} provides some of the results in \cite{BW,I,Ta}. 

Corollary \ref{cc} also applies if  $$X = W^{2, \frac{p}{p-1}}(\Omega)\cap W_{0}^{1,\frac{p}{p-1}}(\Omega), \quad P(u)=\|u\|^{\frac{p}{p-1}}=\int_\Omega |\Delta u|^{\frac{p}{p-1}}$$ and $\Upsilon_\lambda$, $\Gamma$ are as in \eqref{ecc}, with now $1<q<\frac{p}{p-1}<r$ and $\frac{1}{p}+\frac{1}{r}>\frac{N-2}{N}$. These results, with $b \equiv 1$, were established in \cite{E1,E2}.

Further applications of Corollary \ref{cc} will be given in section \ref{spq} to deal with
 \eqref{fcc}, still defined in $X=W_0^{1,p}(\Omega)$, but now with
$$P(u)=\|u\|^p, \quad \Upsilon(u)=\int_\Omega \left( |\nabla u|^q - \lambda |u|^q\right), \quad \mbox{ and } \quad \Gamma(u)=-\int_\Omega b|u|^{r},$$

or $$P(u)=\|u\|^q, \quad \Upsilon(u)=-\int_\Omega b|u|^{r}, \quad \mbox{ and } \quad \Gamma(u)=\int_\Omega \left( |\nabla u|^p - \lambda |u|^p\right),$$
which correspond to $(p,q)$-Laplacian problems with $1<q<p$.

We shall also apply  Corollary \ref{cc} to  \eqref{fcc} with $X:=H_0^1(\Omega)\times H_0^1(\Omega)$, which is a Banach space with the norm given by $\|(u,v)\|^2=\int_\Omega \left(|\nabla u|^2+|\nabla v|^2\right)$, and 
$$P(u,v)=\|v\|_q^q, \quad \Upsilon(u,v)=\|(u,v)\|^2-2\lambda\int_\Omega uv, \quad \mbox{and} \quad \Gamma(u,v)=-\int_\Omega b|u|^r,$$
where $2<q<r<2^*$. This functional is associated to a gradient system, see subsection \ref{kayeexample1}.
	
\medskip

\subsection{Further examples}\label{fe}\strut\\

In the next sections we provide several applications of our results. Before that, let us show that some further established results can be derived from Theorems \ref{thm2}-\ref{tp3}.

Given $p \in (1,\infty)$ we denote by 
\begin{equation}\label{eigenproblem}
\lambda_1(p):=\inf\left\{\int_\Omega |\nabla u|^p: u\in W_0^{1,p}(\Omega), \int_\Omega |u|^p=1\right\}
\end{equation}
the first eigenvalue of the Dirichlet p-Laplacian and by $\phi_1(p)>0$ a first eigenfunction associated to $\lambda_1(p)$. When $p=2$ we simpy write $\phi_1$ and $\lambda_1$.

\medskip
\subsection{A $(p,q)$-Laplacian problem}\strut \\\label{pqexample}

First we give a simple example where Theorems \ref{thm2} and \ref{tp1} apply. 
Consider the $(p,q)$-Laplacian problem \begin{equation}\label{pex}
-\Delta_p u -\Delta_q u = \alpha |u|^{p-2}u+\beta |u|^{q-2}u, \quad u \in W_0^{1,p}(\Omega),
\end{equation}
with $1<q<p$ and $\alpha,\beta \in \Re$. This problem has been recently investigated in  \cite{BT,BT2,BT3}. Regarding the use of the Nehari manifold or the fibering method for $(p,q)$-Laplacian problems, we refer to \cite{BT,BT3,CI,FRQ, LY,PRR,RQ}.

Let 
$$\Phi(u)=\frac{1}{p}\int_\Omega \left(|\nabla u|^p- \alpha|u|^p \right)+ \frac{1}{q}\int_\Omega \left(|\nabla u|^q-\beta |u|^q\right)$$
be defined in  $X=W_0^{1,p}(\Omega)$. Note that $\Phi$ and $J$ given by
$$J(u)=\int_\Omega \left(|\nabla u|^p- \alpha|u|^p \right)+ \int_\Omega \left(|\nabla u|^q-\beta |u|^q\right)$$ are weakly lower semicontinuous.
We apply Theorems \ref{thm2} and \ref{tp1} with $$Y_1=\left\{u \in X: \int_\Omega \left(|\nabla u|^p- \alpha|u|^p \right)<0\right\} $$
and
$$Y_2=\left\{u \in X: \int_\Omega \left(|\nabla u|^q- \beta|u|^q \right)<0\right\},$$
respectively. 

We see that $Y_1$ and $Y_2$ are nonempty if, and only if, $\alpha >\lambda_1(p)$ and $\beta>\lambda_1(q)$, respectively.
We set 
$$
\beta_*(\alpha):=\inf_{u\in \overline{Y}_1 \setminus \{0\}}\frac{\int_\Omega |\nabla u|^q}{\int_\Omega |u|^q} \quad \mbox{and} \quad \alpha_*(\beta):=\inf_{u\in \overline{Y}_2 \setminus \{0\}}\frac{\int_\Omega |\nabla u|^p}{\int_\Omega |u|^p}.
$$

We claim that for any $\beta<\beta_*(\alpha)$ there exists $C>0$ such that
\begin{equation}\label{ipq}
\int_\Omega \left(|\nabla u|^q- \beta|u|^q \right) \geq C\|u\|^q \quad \forall u \in \overline{Y}_1.
\end{equation}
Indeed, on the contrary there would be a sequence $(u_n) \subset Y_1$ such that
$\|u_n\|=1$ and $\int_\Omega \left(|\nabla u_n|^q- \beta|u_n|^q \right) \to 0$. We can assume that $u_n \rightharpoonup u$ in $X$, so that by weak lower semicontinuity, we have $\int_\Omega \left(|\nabla u|^p- \alpha|u|^p \right)\leq 0$ and $\int_\Omega \left(|\nabla u|^q- \beta|u|^q \right)\leq 0$. Moreover, from the first inequality it is clear that $u \neq 0$, and the second one implies that $\beta \geq \beta_*(\alpha)$.\\

In a similar way, one can show that for $\alpha<\alpha_*(\beta)$ there exists $C>0$ such that 
\begin{equation}\label{ipq1}
\int_\Omega \left(|\nabla u|^p- \alpha|u|^p \right) \geq C\|u\|^p \quad \forall u \in \overline{Y}_2.
\end{equation}
\medskip

	\noindent	{\bf Minimization in $\mathcal{N} \cap Y_1$:} 	
Let us take $\alpha >\lambda_1(p)$ and $\beta<\beta_*(\alpha)$.
It is straightforward that (H1) holds for all $u \in Y_1$. 
By \eqref{ipq} we have
$$\Phi(u)=\frac{p-q}{pq}\int_\Omega \left(|\nabla u|^q-\beta |u|^q\right)\geq C\|u\|^q$$
for any $u \in \mathcal{N} \cap Y_1$,
 and
\begin{equation}\label{ij}
J(u) \geq  C\|u\|^q+\int_\Omega \left(|\nabla u|^p-\alpha|u|^p \right) \geq C\|u\|^q-C_1\|u\|^p
\end{equation}
for $u \in Y_1$. It follows that $J>0$ on $\partial Y_1 \setminus \{0\}$ and in
$Y_1 \cap B(0,R)$, for some $R>0$.   Moreover, the first inequality in \eqref{ij} also shows that (HJ) holds for $u=0$.
 From Lemma \ref{l2} we infer that $(HY)_{c_1}$ is satisfied for $c_1:=\inf_{\mathcal{N}\cap Y_1} \Phi$. So Theorem \ref{thm2} yields that $c_1$ is a critical value of $\Phi$. Moreover, if $\beta<\lambda_1(q)$ then  $J>0$ in $X \setminus Y_1$, and consequently $c_1>0$ is the ground state level of $\Phi$.\\

	\noindent{\bf Minimization in $\mathcal{N} \cap Y_2$:} 	 It is clear that (H2) holds for any $u \in Y_2$. Given $\beta >\lambda_1(q)$ and $\alpha<\alpha_*(\beta)$, 
by \eqref{ipq1} we have that $\Phi$ is coercive in $Y_2$. Moreover, if $u \in \partial Y_2$ then $\Phi(u)=\frac{1}{p}\int_\Omega \left(|\nabla u|^p- \alpha|u|^p \right)>0$.  Finally, $\mathcal{N} \setminus Y_2$ is empty for $\alpha\leq \lambda_1(p)$, whereas $\Phi>0$ on $\mathcal{N} \setminus Y_2$  for $\alpha>\lambda_1(p)$. By Theorem \ref{tp1}, we infer that  $c_2:=\inf_{\mathcal{N}\cap Y_2} \Phi<0$ is the ground state level of $\Phi$.\\

 \begin{rmk}\label{pqrem}  {\rm Let $\phi_q=\phi_1(q)$, $\phi_p=\phi_1(p)$, and
		\begin{equation*}
		\alpha_0:=\frac{\int_\Omega |\nabla \phi_q|^p}{\int_\Omega \phi_q^p}, \ \ \ \beta_0:=\frac{\int_\Omega |\nabla \phi_p|^q}{\int_\Omega \phi_p^q}.
		\end{equation*}
		In \cite[Proposition 7]{BT2} the following properties of $\beta_{*}(\alpha)$ are shown, among others:
		$$ \lambda_1(q)<\beta_*(\alpha)<\beta_* \mbox{ for } \alpha\in (\lambda_1(p),\alpha_*) \quad \mbox{and} \quad \beta_*(\alpha)=\lambda_1(q) \mbox{ for } \alpha\ge \alpha_*.$$
		In a similar way, $\alpha_*(\beta)$ satisfies
		$$\lambda_1(p)<\alpha_*(\beta)<\alpha_*\mbox{ for } \beta\in (\lambda_1(q),\beta_*)\quad \mbox{and} \quad \alpha_*(\beta)=\lambda_1(p)\mbox{ for }\beta\ge \beta_*.$$
		It turns out that 
		\begin{equation}\label{rab}
		\alpha_*(\beta_*(\alpha))=\alpha \mbox{ for  } \alpha\in (\lambda_1(p),\alpha^*) \mbox{ and } \beta_*(\alpha_*(\beta) )=\beta \mbox{ for  }\beta\in (\lambda_1(q),\beta_*)
		\end{equation}
		i.e. $\alpha_*(\beta)$ and $\beta_*(\alpha)$ yield the same curve.
		
		Indeed, as shown in the proof of \cite[Proposition 7]{BT2},  any minimizer $u_\alpha$ associated to $\beta_*(\alpha)$ satisfies $\int_\Omega \left(|\nabla u_\alpha|^p- \alpha|u_\alpha|^p \right)= 0$ for $\alpha\in (\lambda_1(p),\alpha^*)$.  It follows that $\alpha_*(\beta_*(\alpha))\le \alpha$. Moreover if $\alpha_*(\beta_*(\alpha))< \alpha$ then there exists $u \in X\setminus \{0\}$ such that $\frac{\int_\Omega |\nabla u|^p}{\int_\Omega |u|^p}<\alpha$ and $\int_\Omega \left(|\nabla u|^q- \beta_*(\alpha)|u|^q \right)\leq 0$.  So $u$ is not a minimizer associated to  $\beta_*(\alpha)$, i.e. 
		$\frac{\int_\Omega |\nabla u|^q}{\int_\Omega |u|^q}>\beta_*(\alpha)$, which provides a contradiction. Therefore
		$\alpha_*(\beta_*(\alpha))=\alpha$ for $\alpha\in (\lambda_1(p),\alpha^*)$, and the second assertion in \eqref{rab} is completely similar.
	}
\end{rmk}

Summing up,  we proved the following result, which yields \cite[Proposition 2]{BT} and \cite[Theorem 2.7]{BT2}:
\begin{cor}
Let $1<p<q$. 
\begin{enumerate}
	\item If $\beta<\lambda_1(q)$  then \eqref{pex} has	 a ground state solution at a positive level for $\alpha >\lambda_1(p)$.
	\item If $\beta>\lambda_1(q)$  then \eqref{pex} has	a ground state solution at a negative level for $\alpha <\alpha_*(\beta)$, and a second solution (which is a ground state solution relative to $Y_1$) for $\lambda_1(p)<\alpha<\alpha_*(\beta)$.\\
\end{enumerate} 
\end{cor}

\medskip

\subsection{A Kirchhoff problem}\strut \\

We consider the Kirchhoff equation
\begin{equation}\label{pk2}
-\left(a+b\int_{\Omega}|\nabla u|^2 dx\right)\Delta u= \lambda u+\mu |u|^2u, \quad u \in H_0^1(\Omega),
\end{equation}
where $a,b>0$, and $\lambda,\mu \in \Re$. Kirchhoff type problems have been largely investigated, mostly bt variational methods. We refer to \cite{CO,CKW,FRQ,LTW,SS} for works on this class of problems via the Nehari manifold method.

 Let	us set, for $u \in X=H_0^1(\Omega)$,
$$\Phi_1(u):=a\| u\|^2-\lambda \|u\|_2^2\quad \mbox{and} \quad \Phi_2(u):=b\|u\|^4-\mu \|u\|_4^4.$$
The energy functional for \eqref{pk2} is given by $
\Phi(u)=\frac{1}{2}\Phi_1(u)+\frac{1}{4}\Phi_2(u)$,
so that $J=\Phi_1+\Phi_2$.
As mentioned in the introduction, this problem has a variational structure similar to the one of the $(p,q)$ Laplacian problem in the previous subsection.

Let $\psi_1$ be the unique positive minimizer achieving (see \cite{D})
$
\mu_1=\inf\left\{\frac{\|u\|^4}{\|u\|_4^4}:u\in X\setminus\{0\}\right\}
$. Set $\lambda^*=a\frac{\|\psi_1\|^2}{\|\psi_1\|_2^2}$ and $\mu^*=b\frac{\|\phi_1\|^4}{\|\phi_1\|_4^4}$,
and note that $\lambda^*>a\lambda_1$ and $\mu^*>b\mu_1$ (see \cite{CO} or \cite[Lemma 3]{SS}).

We shall deal with the cones 	$$Y_1:=\left\{u\in X: \Phi_2(u)<0  \right\} \quad \mbox{and} \quad Y_2:=\left\{u\in X: \Phi_1(u)<0  \right\}$$ and the extremal parameter (see \cite{I1})
$$
\mu^*(\lambda):=\inf\left\{b\frac{\|u\|^4}{\|u\|_4^4}: u \in \overline{Y}_2\setminus \{0\}\right\}.
$$
By \cite[Proposition 3]{SS} we have $\mu^*(\lambda)\in(b\mu_1,\mu^*)$ for all $\lambda\in(a\lambda_1,\lambda^*)$.

	For each $\mu>b\mu_1$, define
	$$
	\lambda^*(\mu):=\inf\left\{a\frac{\|u\|^2}{\|u\|_2^2}: u \in \overline{Y}_1\setminus\{0\} \right\}.
	$$
	Similar to Remark \ref{pqrem}, by using \cite[Theorem 6]{SS} and the ideas of \cite{BT2} one can prove that $\lambda^*(\mu)$ and $\mu^*(\lambda)$ yield the same curve. From now on we assume this fact.\\

	\noindent {\bf Minimization in $\mathcal{N} \cap Y_1$:}	
		Let us prove that Theorem \ref{thm2} can be applied for $(\lambda,\mu)\in A:=[a\lambda_1,\lambda^*)\times(b\mu_1,\mu^*(\lambda))\cup(-\infty,a\lambda_1)\times(b\mu_1,+\infty)$. From \cite[Proposition 4]{SS}, it follows that $Y_1$ is non-empty and one can easily see that (H1) is satisfied on $Y_1$. Now we claim that
		\begin{equation}\label{kirkk1}
		\Phi_1(u)\ge C\|u\|^2 \quad \forall u\in \overline{Y}_1.
		\end{equation}
		Otherwise we can find a sequence $(u_n)\in S\cap \overline{Y}_1$ such that $\Phi_1(u_n)< 1/n$. We can assume that $u_n \rightharpoonup u$, so that $\Phi_1(u_0)\le 0$ and $\Phi_2(u)\le 0$. Moreover, it is clear that $u \not \equiv 0$, which contradicts $(\lambda,\mu)\in B$ and thus \eqref{kirkk1} holds true.
		
		Next we prove that $(HY)_{c_1}$ is satisfied for $c_1:=\inf_{\mathcal{N}\cap Y_1} \Phi$: indeed note by \eqref{kirkk1} that
		\begin{equation*}
		J(u)\ge C\|u\|^2+\Phi_2(u)\ge C\|u\|^2-C_1\|u\|^4, \quad \forall u\in \overline{Y}_1.
		\end{equation*}
		It follows that $J>0$ on $\partial Y \setminus \{0\}$ and in
		$Y \cap B(0,R)$, for some $R>0$.  Moreover, the first inequality also shows that (HJ) holds with $u=0$.
		From Lemma \ref{l2} we infer that $(HY)_{c_1}$ holds, so Theorem \ref{thm2} yields that $c_1$ is a critical value of $\Phi$. Moreover, if $\lambda<a\lambda_1$ then  $J>0$ in $X \setminus Y_1$, so that $c_1>0$ is the ground state level of $\Phi$.\\

		\noindent {\bf Minimization in $\mathcal{N} \cap Y_2$:}
		We claim that Theorem \ref{tp1} can be applied for $(\lambda,\mu)\in B:=(a\lambda_1,\lambda^*)\times[b\mu_1,\mu^*(\lambda))\cup(a\lambda_1,+\infty)\times(-\infty,b\mu_1)$. From \cite[Proposition 4]{SS}, it follows that $Y_2$ is non-empty for all $(\lambda,\mu)\in B$ and (H2) is satisfied therein. Let us  prove that there exists a positive constant $C$ such that
		\begin{equation}\label{kirkk}
		\Phi_2(u)\ge C\|u\|^4,\quad \forall u\in \overline{Y}_2.
		\end{equation}
		Otherwise we can find a sequence $(u_n)\in S\cap Y_2$ such that $\Phi_2(u_n)< 1/n$. We can assume that $u_n \rightharpoonup u$, so that $\Phi_1(u)le 0$ and $\Phi_2(u)\le 0$. Moreover, it is clear that $u \not \equiv 0$, which contradicts $(\lambda,\mu)\in B$ and \eqref{kirkk} holds true. Thus $\Phi$ is coercive in $Y_2$. Still from \eqref{kirkk} we have $\Phi=\frac{1}{4}\Phi_2>0$  on  $\partial Y_2$.
		Since $\Phi=\frac{1}{4}\Phi_1\ge 0$ on $\mathcal{N}\setminus Y_2$ we conclude from Theorem \ref{tp1} that $c_2:=\inf_{\mathcal{N} \cap Y_2} \Phi=\inf_{Y_2} \Phi=\inf_{\mathcal{N}} \Phi$ is negative and achieved by a local minimizer of $\Phi$. \\

Hence we obtain the following result \cite[Theorem 2]{SS}:
\begin{cor}\strut
	\begin{enumerate}
		\item Assume either $\lambda<a \lambda_1$ and $\mu >b \mu_1$ or $a \lambda_1\le \lambda<\lambda^*$ and $b \mu_1 < \mu < \mu^*(\lambda)$. Then \eqref{pk2} has a  solution which is a ground state relative to $Y_1$. Moreover, this one is a ground state solution at a positive level for $\lambda<a \lambda_1$.
		\item Assume either $\lambda>a \lambda_1$ and $\mu <b \mu_1$ or $a \lambda_1< \lambda<\lambda^*$ and $b \mu_1 \le \mu < \mu^*(\lambda)$. Then \eqref{pk2} has a ground state  solution, which is a local minimizer of the energy functional, at a negative level.
	\end{enumerate}	
\end{cor}

\medskip




\medskip

\section{Applications to pdes} 
\medskip
Let us obtain now some new results for $(p,q)$-Laplacian problems and Kirchhoff equations.
Throughout the next sections we assume that $b \in L^{\infty}(\Omega)$, and $g:\Omega \times \Re \to \Re$ and $f:\Re \to \Re$ are continuous. We set $G(x,t):=\displaystyle\int_{0}^{t}g(x,s)\ ds$ and $F(t):=\displaystyle\int_{0}^{t}f(s)\ ds$ for $t \in \Re$.\\

\subsection{Generalized $(p,q)$-Laplacian problems}\strut\\
\label{spq}

Let $a:[0,\infty)\rightarrow
[0,\infty)$ be a $\mathcal{C}^{1}$ function satisfying the following conditions:\\
\begin{itemize}
	\item[(A1)] $k_0\left( 1+t^{\frac{q-p}{p}}\right) \leq a(t) \leq k_1\left( 1+t^{\frac{q-p}{p}}\right)$ for every  $t>0$, and some constants $k_0,k_1>0$ and $p\geq q>1$.
	\item[(A2)] $a$ is non-increasing.
	\item[(A3)]  $t\mapsto a(t^p)t^p$ and $t \mapsto A(t^p)$ are convex in  $(0,\infty)$.\\
\end{itemize}
The problem
\begin{equation}\label{quasi}
-div \left(a(|\nabla u|^p)|\nabla u|^{p-2}\nabla u\right) = g(x,u), \quad u \in W_0^{1,p}(\Omega),
\end{equation}
is a generalized $(p,q)$-Laplacian problem, as the operator in the left-hand side reduces to $-\Delta_p -\Delta_q$ if we choose $a(t)=1+t^{\frac{q-p}{p}}$, which is one of the main prototypes satisfying (A1)-(A3).

Let us set $A(t):=\displaystyle\int_{0}^{t}a(s)\ ds$ for $t \ge 0$.

\begin{cor} \label{c2}
	Assume  that $g(x,t)=0$ for every $x \in \Omega \setminus \Omega'$, where $\Omega'$ is an open subset of $\Omega$. Moreover, assume that:
	\begin{enumerate}
		
		\item $\dis \lim_{|t| \rightarrow 0} \frac{g(x,t)}{|t|^{q-1}}=0$ uniformly for $x \in \Omega'$.
		\item There exists $r \in (p,p^*)$ such that $\dis \lim_{|t| \rightarrow \infty} \frac{g(x,t)}{|t|^{r-1}}=0$ uniformly for $x \in \Omega'$.
		\item For every $x \in \Omega'$ the map $t \mapsto \frac{g(x,t)}{|t|^{p-1}}$
		is increasing on $\Re \setminus \{0\}$.
		\item $\dis \lim_{|t| \to \infty}\frac{G(x,t)}{|t|^p}=\infty$ uniformly for  $x \in \Omega'$.
	\end{enumerate}
	Then \eqref{quasi} has a positive ground state.
\end{cor}

\begin{proof}
	Let $X=W_0^{1,p}(\Omega)$  with $\|u\|=\left(\int_\Omega |\nabla u|^p\right)^{\frac{1}{p}}$, and $\Phi:=I_0-I$, where
	\begin{equation}
	I_0(u):=\frac{1}{p}\displaystyle\int_{\Omega}A(|\nabla u|^p)\quad \mbox{and} \quad I(u):=\int_\Omega G(x,u).
	\end{equation}
	for $u \in X$. Note that
	$$I_0'(u)u=\int_\Omega a(|\nabla u|^p) |\nabla u|^p\quad \mbox{and}\quad I'(u)v= \int_\Omega g(x,u)v.$$
	(A3) and (2) provide the weak lower semicontinuity of $I_0$ and $u \mapsto I_0'(u)u$, and the strong continuity of $I'$, respectively. (A1) implies that $I'_0(u)u \geq  k_0\|u\|^p$, whereas (1) and (2) imply that $J(u) \geq C_1\|u\|^p-C_2\|u\|^r$ (see the proof of \cite[Corollary 2.1]{FRQ}).
	We set 
	$$Y:=\{u \in X: u \not \equiv 0 \text{ in } \Omega'\},$$
	which is clearly an open cone in $X$.
	We also see that $\partial Y=X \setminus Y=\{u \in X: u  \equiv 0 \text{ in } \Omega'\}$ so that $I(u)=I'(u)u=0$ for every  $u \in \partial Y=X \setminus Y$.

	From (A1) we have
	$$\frac{I_0(tu)}{t^p}\leq \frac{k_1}{q}t^{q-p} \int_\Omega |\nabla u|^q+\frac{k_1}{p}\|u\|^p \quad \mbox{and}\quad \frac{I(tu)}{t^p}=\int_\Omega  \frac{G(x,tu)}{t^p}\to \infty$$
	uniformly for $u$, in a weakly compact subset of $Y$. 
	Finally, by (A2) and (3) the maps
	$$t \mapsto \frac{I_0'(tu)u}{t^{p-1}}=\int_\Omega a\left(t^p|\nabla u|^p\right) |\nabla u|^p \quad \mbox{and} \quad t \mapsto \frac{I'(tu)u}{t^{p-1}}=\int_{\Omega'} \frac{g(x,tu)}{t^{p-1}}u$$ are nonincreasing and increasing, respectively, for every $u \in Y$.  Corollary \ref{c1} and Remark \ref{rc1} yield the conclusion.\\
\end{proof}

In particular, for $g(x,u)=b(x)f(u)$ we obtain the following result:

\begin{cor}
	Assume that $b \geq 0$, and $f$ is such that:
	\begin{enumerate}
		
		\item $\dis \lim_{t \rightarrow 0} \frac{f(t)}{t^{q-1}}=0$.
		\item $\dis \lim_{t \rightarrow \infty} \frac{f(t)}{t^{r-1}}=0$ for some  $r \in (p,p^*)$.
		\item  $t \mapsto \frac{f(t)}{|t|^{p-1}}$
		is increasing on $\Re \setminus \{0\}$ and goes to infinity as $t \to \infty$.
	\end{enumerate}
	Then the problem \begin{equation}
	-div \left(a(|\nabla u|^p)|\nabla u|^{p-2}\nabla u\right) = b(x)f(u), \quad u \in W_0^{1,p}(\Omega),
	\end{equation} has a positive ground state level.\\
\end{cor}

\begin{cor}\label{c21}
	Let $b_1,b_2 \in L^{\infty}(\Omega)$ with $b_1,b_2 \geq 0$ and $b_1b_2 \equiv 0$, and $r_1 \in (p,p^*)$, $r_2 \in (q,p^*)$ with $r_1 \geq r_2$. Then the problem \begin{equation}
	-div \left(a(|\nabla u|^p)|\nabla u|^{p-2}\nabla u\right) = b_1(x)|u|^{r_1-2}u - b_2(x)|u|^{r_2-2}u, \quad u \in W_0^{1,p}(\Omega),
	\end{equation}
	has a positive ground state level.
\end{cor}

\begin{proof}
	Let $X=W_0^{1,p}(\Omega)$ and  \[
	Y=
	\begin{cases}
	\{u \in X: \int_\Omega b_1|u|^{r_1}>0\} & \mbox{ if } r_1 > r_2,\\
	\{u \in X: \int_\Omega (b_1-b_2)|u|^{r}>0\} & \mbox{ if }r_1 = r_2=r.
	\end{cases}
	\]
	If $r_1>r_2$ then we  apply Corollary \ref{c1} with 
	\begin{equation}
	I_0(u):=\frac{1}{p}\displaystyle\int_{\Omega}A(|\nabla u|^p)+\int_\Omega b_2|u|^{r_2}\quad \mbox{and} \quad I(u):=\int_\Omega b_1|u|^{r_1}.
	\end{equation}
	Indeed, since $b_2 \geq 0$ we see that Corollary \ref{c1} applies with $\eta=p$ and $\sigma=r_2$.
	If $r_1=r_2=r$ then  Corollary \ref{c1} applies with 
	$$
	I_0(u):=\frac{1}{p}\displaystyle\int_{\Omega}A(|\nabla u|^p)\quad \mbox{and} \quad I(u):=\int_\Omega b|u|^{r},
	$$
	where $b=b_1-b_2$. Since $p<r<p^*$ we see that conditions (2) and (4) of Corollary \ref{c1} hold with $r=\sigma=p$. 
\end{proof}

\begin{cor}\label{c22}
	Let $b^+ \not \equiv 0$ and $r \in (p,p^*)$. Then the problem \begin{equation}
	-div \left(a(|\nabla u|^p)|\nabla u|^{p-2}\nabla u\right) = b(x)|u|^{r-2}u, \quad u \in W_0^{1,p}(\Omega),
	\end{equation}
	has a positive ground state level.
\end{cor}

\begin{proof}
	It follows from the previous corollary with $\alpha=\beta$, $b_1=b^+$ and $b_2=b^-$.\\
\end{proof}
\medskip

Next we deal with the $(p,q)$-Laplacian problem 
	\begin{equation}\label{pqr}
	-\Delta_pu- \Delta_q u=\lambda |u|^{\theta-2}u+b(x)|u|^{r-2}u, \quad u\in W_0^{1,p}(\Omega).\\
	\end{equation}
	\medskip
where $1<q<p$,  $\lambda \in \Re$, and $1<\theta,r<p^*$.
 If $r<q$ then we set
\begin{equation*}
\lambda^*=\inf\left\{\alpha(u):  u\in X, \int b|u|^r>0\right\},
\end{equation*}
where
\begin{equation*}
\alpha(u)=\frac{\int |\nabla u|^p}{\int|u|^p}+\frac{q-r}{p-q}\left(\frac{p-q}{p-r}\right)^{\frac{p-r}{q-r}}\frac{\left( \int |\nabla u|^q\right)^{\frac{p-r}{q-r}}}{\int |u|^p\left(\int b|u|^r\right)^{\frac{p-q}{q-r}}}
\end{equation*}
and, if $r>p$,
\begin{equation*}
\lambda^{**}=\inf\left\{\beta(u): u\in X, \int b|u|^r>0 \right\},
\end{equation*}
where
\begin{equation*}
\beta(u)=\frac{\int |\nabla u|^q}{\int|u|^q}+\frac{r-p}{r-q}\left(\frac{p-q}{r-q}\right)^{\frac{p-q}{r-p}}\frac{\left(\int |\nabla u|^p\right)^{\frac{r-q}{r-p}}}{ \int |u|^q\left(\int b|u|^r\right)^{\frac{p-q}{r-p}}}.
\end{equation*}

Both $\alpha(u)$ and $\beta(u)$ can be obtained as the unique solution, with respect to the variable $(t,\lambda)$, of the equations $\varphi_u'(t)=\varphi_u''(t)=0$. They are called nonlinear Rayleigh's quotients and $\lambda^*,\lambda^{**}$ are called extremal parameters (see \cite{I1}).\\

\begin{cor}
Let $b^+ \not \equiv 0$ and $1<r<q<p=\theta$. 
\begin{enumerate}
\item If $\lambda\le \lambda_1(p)$ then \eqref{pqr} has a ground state solution at a positive level.
\item If $\lambda_1(p)<\lambda<\lambda^*$ then \eqref{pqr} has a a ground state solution at a negative level, and a second solution.
\end{enumerate}
\end{cor}

\begin{proof}
The energy functional is given by
\begin{equation*}
\Phi(u)=\frac{1}{p}\left(\|u\|^p-\lambda  \|u\|_p^p\right)+\frac{1}{q}\|\nabla u\|_q^q-\frac{1}{r}\int b|u|^r,
\end{equation*}
so that
$J(u)=\| u\|^p-\lambda  \|u\|_p^p+ \|\nabla u\|_q^q-\int_\Omega b|u|^r$.

Let us first show that $\Phi$ is coercive if $\lambda \leq \lambda_1(p)$. This is clear if $\lambda < \lambda_1(p)$. Let $\lambda = \lambda_1(p)$ and $(u_n) \subset X$ with $\|u_n\| \to \infty$. If $\Phi(u_n)$ is bounded from above then we can assume that $v_n:=\frac{u_n}{\|u_n\|} \rightharpoonup v_0$ and, since $p>q>r$, we deduce that
$$\| v_n\|^p-\lambda_1(p)  \|v_n\|_p^p \to 0 \quad \mbox{and} \quad \|\nabla v_n\|_q \to 0.$$
The second assertion yields $v=0$, whereas the first one shows that $v_n \to 0$ in $X$, a contradiction. Therefore $\Phi$ is coercive and has a negative global minimum.

 Let us take now $\lambda_1(p)<\lambda<\lambda^*$.
We apply Corollary \ref{cc}  with
$$P(u)=\|\nabla u\|_q^q, \quad \Upsilon(u)=-\int_\Omega b(x)|u|^r, \quad \mbox{and} \quad \Gamma(u)=\| u\|^p-\lambda  \|u\|_p^p.$$
Hence $$Y_1=\left\{u\in X: \int_\Omega b(x)|u|^r>0 \right\} \quad \mbox{and} \quad Y_2=\left\{u\in X: \| u\|^p-\lambda  \|u\|_p^p<0\right\}.$$
We claim that $P(u)\geq C\|u\|^q$ holds along any minimizing sequence for $c$. To this end, we use \cite[Lemma 9]{T} and Poincar\'e's inequality, which implies the desired inequality in $X_k:=\left\{u\in X: \| u\|^p- k \|u\|_p^p<0\right\}$, for any $k>0$. Since $c<0$ it is enough to show that for a given $\varepsilon>0$ we have $\{u \in \mathcal{N}: \Phi(u) \leq -\varepsilon\} \subset X_k$ for some $k>0$. On the contrary, we find a sequence $(u_n) \subset \mathcal{N}$ such that $\| u_n\|^p-n  \|u_n\|_p^p\ge 0$ and $\Phi(u_n) \leq -\varepsilon$ for every $n$. Thus $\|u_n\|_p^p \leq \frac{1}{n}\|u_n\|^p$ and consequently
$$\left(1-\frac{\lambda}{n}\right) \|u_n\|^p\leq \| u_n\|^p-\lambda \|u_n\|_p^p<\int_\Omega b|u_n|^r\leq C\|u_n\|^r.$$
Since $p>r$ we deduce that $(u_n)$ is bounded, so we may assume that $u_n \rightharpoonup u$ in $X$. But $\|u_n\|_p^p \leq \frac{1}{n}\|u_n\|^p$ implies that $u=0$, which contradicts $\Phi(u)\leq \liminf \Phi(u_n) \leq -\varepsilon$. 
Therefore the claim is proved and the inequality $P(u)\geq C\|u\|^q$ also holds in $Y_2$. Note also that $J$ satisfies (HJ) since $\Gamma$ does so.
Finally, (H5) holds since $\lambda<\lambda^*$, so Corollary \ref{cc} yields the desired conclusion.

\end{proof}

\begin{cor}
	Let $b^+ \not \equiv 0$ and $1<q=\theta<p<r<p^*$. 
	\begin{enumerate}
		\item If $\lambda\le \lambda_1(q)$ then \eqref{pqr} has a ground state solution at a positive level.
		\item If $\lambda_1(q)<\lambda<\lambda^{**}$ then \eqref{pqr} has a ground state solution at a negative level, and a second solution.
	\end{enumerate}
\end{cor}

\begin{proof}

We consider the functional
\begin{equation*}
\Phi(u)=\frac{1}{p}\|u\|^p+\frac{1}{q}\left(\| \nabla u\|_q^q-\lambda  \|u\|_q^q\right)-\frac{1}{r}\int_\Omega b|u|^r,
\end{equation*}
and the open cones $$Y_1=\left\{u \in X: \| \nabla u\|_q^q-\lambda  \|u\|_q^q<0\right\}\quad
\text{ and }
\quad Y_2=\left\{u\in X: \int_\Omega b(x)|u|^r>0 \right\}.$$
Note that
$
J(u)=\|u\|^p+\| \nabla u\|_q^q-\lambda  \|u\|_q^q-\int_\Omega b|u|^r
$
and
$$\Phi(u)=\frac{r-p}{rp}\|u\|^p +\frac{r-q}{rq}\left(\| \nabla u\|_q^q-\lambda  \|u\|_q^q\right) \geq C_1\|u\|^p -C_2\|u\|^q, \ \forall u \in \mathcal{N},$$
i.e. $\Phi$ is coercive on $\mathcal{N}$.\\

For $\lambda \leq \lambda_1(q)$ we apply Corollary \ref{c1} with $Y=Y_2$. Indeed, note that $\Phi=I_0-I$ with 
$$I_0(u)=\frac{1}{p}\|u\|^p+\frac{1}{q}\left(\| \nabla u\|_q^q-\lambda  \|u\|_q^q\right)\quad
\text{and}
\quad I(u)=\frac{1}{r}\int b|u|^r.$$
Since $\lambda \leq \lambda_1(q)$ we have $I_0'(u)u \geq \|u\|^p$ for all $u \in X$.
Thus conditions (1)-(4) of Corollary \ref{c1} clearly hold. Moreover, $I'(u)u \leq 0$ for $u \in X \setminus Y_2$, 
so $c_2:=\displaystyle \inf_{\mathcal{N}\cap Y_2}\Phi$ is the ground state level of $\Phi$.\\

Finally,  one may easily see that  $$P(u)=\|u\|^p, \quad \Upsilon(u)=\Upsilon_\lambda(u)=\| \nabla u\|_q^q-\lambda  \|u\|_q^q \quad \mbox{ and } \quad \Gamma(u)=-\int_\Omega b|u|^{r}$$
satisfy the conditions of Corollary \ref{cc} for any $\lambda_1(q)<\lambda<\lambda^{**}$. Indeed, in this case $Y_1,Y_2$ are nonempty since $\lambda>\lambda_1(q)$ and $b^+ \not \equiv 0$, respectively. Moreover, 
(H5) is satisfied since $\lambda<\lambda^{**}$.
Thus $c_1:=\displaystyle \inf_{\mathcal{N}\cap Y_1}\Phi$ is the ground state level of $\Phi$, which has a second critical point that belongs to $Y_2$.
\end{proof}
\medskip
\subsection{Problems involving a $(p,q)$-Laplacian operator with spatial dependence}\strut\\

Let $D,E \subset \Omega$ be two disjoint smooth subdomains with $\overline{D \cup E}=\overline{\Omega}$.
We consider the problem 
\begin{equation}\label{mr}
-\chi_D \Delta u - \chi_E \Delta_p u=b(x)|u|^{r-2} u, \quad u \in X,
\end{equation}
where  $X=\{u \in H_0^1(\Omega): \nabla u \in L^p(E)\}$ and $\chi$ is the characteristic function.

 This problem has been investigated in \cite{MR} with $b \equiv \lambda>0$ and $2<r<p$. 
 We shall deal here with the case $2<p<r<2^*$.
 
 \begin{cor}
Let $b^+ \not \equiv 0$ and $2<p<r<2^*$. Then \eqref{mr} has a ground state solution at a positive level.
 \end{cor}

\begin{proof}
 We set $\|u\|=\|\nabla u\|_{2,D}+\|\nabla u\|_{p,E}$ and observe that $X$ equipped with this norm is a reflexive Banach space. Moreover, $\|\cdot \|$ is equivalent to the norm given by $\|\nabla u\|_{2,\Omega}+\|\nabla u\|_{p,E}$, see \cite[Lemma 2.1]{MR}.

The energy functional is given by
$$\Phi(u)=\frac{1}{2}\int_D |\nabla u|^2 +\frac{1}{p}\int_E |\nabla u|^p - \frac{1}{r}\int_\Omega b(x)|u|^r,$$
for $u \in X$.
Let $Y=\left\{u\in X: \int_\Omega b(x)|u|^r>0 \right\}$.
The condition $r>p>2$ clearly implies (H1) for any $u \in Y$.
Note also that if $p<\sigma<r$ then
$$\frac{\Phi(tu)}{t^\sigma} =-t^{r-\sigma}\int_\Omega b(x)|u|^r+o(1)\to -\infty$$ as $t \to \infty$, uniformly for $u$ on weakly compact subsets of $Y$.
Moreover,
$\Phi=I_0-I$ with
$I_0(u)=\frac{1}{2}\int_D |\nabla u|^2 +\frac{1}{p}\int_E |\nabla u|^p$ and $I(u)=\frac{1}{r}\int_\Omega b(x)|u|^r$, which is weakly continuous since $r<2^*$. If $u \in  S$ then, for $t>0$ large enough, we have
$$I_0(tu)\geq \frac{t^2}{2} \left(\|\nabla u\|_{2,D}^2+\|\nabla u\|_{p,E}^p\right)\geq Ct^2\|u\|^p=Ct^2$$
where we have used the inequality 
$ (x^2+y^p)^{\frac{1}{p}}\geq C(x+y)$, which holds for $x,y \in [0,1]$ and some $C>0$.
Finally, it is clear that $J>0$ on $\partial Y$ and since $\|\cdot \|$ is equivalent to $\|\nabla u\|_2+\|\nabla u\|_{p,E}$, we see that $J>0$
on $B(0,R) \setminus \{0\}$ if $R>0$ is small enough. Lemma \ref{l1} and Theorem \ref{thm2} yield the desired conclusion.
\end{proof}
\medskip
\subsection{Kirchhoff problems}\strut\\

Let us deal now with the problem
\begin{equation}\label{pk}
-M\left(\int_{\Omega} |\nabla u|^2\right) \Delta u =g(x,u), \quad u \in H_0^1(\Omega).
\end{equation}
Throughout this subsection $M:[0,\infty)\rightarrow
[0,\infty)$ is an increasing ${C}^{1}$ function such that
 $M(0):=m_0>0$ and  $t\mapsto\displaystyle\frac{M(t)}{t}$ is decreasing in $(0,\infty)$.
We set  $\hat{M}(t):=\displaystyle\int_{0}^{t}M(s)\ ds$ for $t \ge 0$. 

\begin{cor}\label{c3}
	Assume that $g(x,t)=0$ for every $x \in \Omega \setminus \Omega'$, where $\Omega'$ is an open subset of $\Omega$. Moreover, assume that:
	
	\begin{enumerate}
		\item $\dis \lim_{t \rightarrow 0} \frac{g(x,t)}{t}=0$ uniformly for $x \in \Omega'$.
		\item $\dis \lim_{t \rightarrow \infty} \frac{G(x,t)}{t^4}=\infty $ uniformly for $x \in \Omega'$.
		\item There exists $C>0$ and $r \in (4,6)$ such that $|g(x,t)|\le C (1+|t|^{r-1})$ for any $x \in \Omega'$ and $t \in \Re$.
		\item For any $x \in \Omega'$ the map $t\mapsto  \frac{g(x,t)}{t^{3}}$ is increasing.
	\end{enumerate}
	Then \eqref{pk}	has a positive ground state level.
\end{cor}

\begin{proof}
The energy functional is given by $\Phi=I_0-I$, where
$$I_0(u):=\frac{1}{2}\hat{M}\left(\|u\|^2\right) \quad
\text{and}
\quad I(u):=\int_\Omega G(x,u),$$
 for $u \in X=H_0^1(\Omega)$.
 Note also that $$I_0'(u)u=M(\|u\|^2)\|u\|^2 \quad
 \text{and}
 \quad I'(u)u= \int_\Omega g(x,u)u.$$ Since $\hat{M}$ is convex, we see that $I_0$ and $u \mapsto I_0'(u)u$ are weakly lower semicontinuous, whereas (3) implies that $I'$ is strongly continuous. Since $M$ is increasing we have 
 $I_0'(u)u\geq m_0\|u\|^2$ for every $u \in X$.
 
Let $Y:=\{u \in X: u \not \equiv 0 \text{ in } \Omega'\}$.
Note that $I(u)=I'(u)u=0$ for every $u \in Y$.
Since $\frac{M(t)}{t}$ is decreasing we find that $M(t)\leq C(1+t)$ for some $C>0$ and all $t\geq 0$. It follows that
$\frac{I_0(tu)}{t^4}\leq C_1\|u\|^4+C_2t^{-2}\|u\|^2$, while
$\frac{I(tu)}{t^4}=\int_\Omega \frac{G(x,tu)}{t^4} \to \infty$
as $t \to \infty$, uniformly on weakly compact subsets of $Y$. Since $M$ and $t\mapsto  \frac{g(x,t)}{t^{3}}$ are increasing, the maps $t \mapsto \frac{I_0'(tu)u}{t^{\sigma-1}}$ and $t \mapsto \frac{I'(tu)u}{t^{\sigma-1}}$ are increasing.  

Finally, arguing as in the proof of \cite[Corollary 2.4]{FRQ} one can show that $\liminf_{u \to 0} \frac{\Phi'(u)u}{\|u\|^2}>0$.
Corollary \ref{c1} and Remark \ref{rc1} yield the desired conclusion.\\
\end{proof}

The next results are similar to Corollaries \ref{c21} and \ref{c22}, so we omit their proof:

\begin{cor}
	Let $b_1,b_2 \in L^{\infty}(\Omega)$ with $b_1,b_2 \geq 0$ and $b_1b_2 \equiv 0$, and $r_1 \in (4,6)$, $r_2 \in (2,6)$ with $r_1 \geq r_2$. Then the problem \begin{equation}
	-M\left(\int_{\Omega} |\nabla u|^2\right) \Delta u = b_1(x)|u|^{r_1-2}u - b_2(x)|u|^{r_2-2}u, \quad u \in H_0^1(\Omega),
	\end{equation}
	has a positive ground state level.
\end{cor}

\begin{cor}
Let $b^+ \not \equiv 0$  and $r \in (4,6)$.	Then the problem \begin{equation}
-M\left(\int_{\Omega} |\nabla u|^2\right) \Delta u =b(x)|u|^{r-2}u, \quad u \in H_0^1(\Omega),
\end{equation}
has a positive ground state level.
\end{cor}
\medskip

\section{Applications to systems of pdes}
\medskip

This final section is devoted to the applications of our results to some elliptic systems. In this regard, let us mention that the Nehari manifold and the fibering method have been extensively exploited. We refer to \cite{AH,BM,BW1,SM,Ws} for results on gradient type systems, and to \cite{BMR,E1,E2,S} for strongly coupled systems that can be reduced to a fourth-order equation.

Recall that $b \in L^{\infty}(\Omega)$,  $g:\Omega \times \Re \to \Re$ and $f:\Re \to \Re$ are continuous, and  $G(x,t):=\displaystyle\int_{0}^{t}g(x,s)\ ds$ and $F(t):=\displaystyle\int_{0}^{t}f(s)\ ds$ for $t \in \Re$.\\
\subsection{A semilinear gradient system}\label{kayeexample1} \strut\\

Consider the system of equations

\begin{equation}\label{pk3}
\left\{
\begin{array}{ll}
-\Delta u = \lambda v+b(x)|u|^{r-2}u \
& \mbox{in} \ \ \Omega, \ \ \\
-\Delta v =  \lambda u-|v|^{q-2}v \
&\mbox{in} \ \ \Omega, \ \ \\
u = v = 0 \ &\mbox{on} \ \
\partial\Omega,
\end{array}
\right.
\end{equation}
where $\lambda \in \Re$ and $2<q<r<2^*$.
We deal with  $X:=H_0^1(\Omega)\times H_0^1(\Omega)$, endowed with the norm $\|(u,v)\|=\left(\|\nabla u\|_{2}^2+\|\nabla v\|_{2}^2\right)^{\frac{1}{2}}$. We also set $\|(u,v)\|_2=\left(\| u\|_{2}^2+\| v\|_{2}^2\right)^{\frac{1}{2}}$. 

Let us set
\begin{equation*}
\lambda^*:=\inf\left\{\lambda(u,v): \int_\Omega b|u|^r>0,\ \int_\Omega uv>0 \right\},
\end{equation*}
where
\begin{equation*}
\lambda(u,v):=\frac{1}{2 \int uv} \left(\|(u,v)\|^2+ C_{r,q}\frac{\|v\|_q^{q\frac{r-2}{r-q}}}{\left(\int_\Omega b|u|^r dx\right)^{\frac{q-2}{r-q}}}\right),
\mbox{ and } 
C_{r,q}:=\frac{r-q}{r-2}\left(\frac{q-2}{r-2}\right)^{\frac{q-2}{r-q}}.
\end{equation*}
It is not hard to show that $\lambda^*$ is achieved, and consequently $\lambda^*>\lambda_1$.

\begin{cor}\label{ws}\strut
	\begin{enumerate}
		\item If $\lambda<\lambda_1$ then \eqref{pk3} has a ground state solution (at a positive level).
		\item If $\lambda\in(\lambda_1,\lambda^*)$ then \eqref{pk3} has a ground state solution, which is a local minimizer (at a negative level), and a second critical point.
	\end{enumerate}	
\end{cor}

\begin{proof}
The energy functional associated to \eqref{pk3} is given by 
$$
\Phi(u,v)= \frac{1}{2}\|(u,v)\|^2-\lambda\int_\Omega uv 
+\frac{1}{q}\|v\|_q^q-\frac{1}{r}\int_\Omega b|u|^r, \quad \mbox{ for } (u,v)\in X.
$$
Thus
$$J(u,v)=\|(u,v)\|^2-2\lambda \int_\Omega uv 
+\|v\|_q^q-\int_\Omega b|u|^r.$$

We introduce the cones
$$Y_1=\left\{(u,v)\in X:\|(u,v)\|^2-2\lambda\int_\Omega uv<0\right\}\quad \mbox{and} \quad Y_2=\left\{(u,v)\in X: \int_\Omega b|u|^r> 0\right\}.$$

\begin{enumerate}
\item  Let  $\lambda<\lambda_1$.	
Since $2\int_\Omega uv \leq \|u\|_2^2+\|v\|_2^2$ we see that 
there exists $C>0$ such that 
\begin{equation}\label{K0}
\|(u,v)\|^2-2\lambda\int_\Omega uv\ge C\|(u,v)\|^2,\ \quad \forall (u,v)\in X. 
\end{equation}
Thus Corollary \ref{c1} applies with $Y=Y_2$,
$$I_0(u):= \frac{1}{2}\|(u,v)\|^2-\lambda\int_\Omega uv 
+\frac{1}{q}\|v\|_q^q\quad \mbox{and} \quad I(u):=\frac{1}{r}\int_\Omega b|u|^r.$$

\item	Let us take now  $\lambda_1<\lambda<\lambda^*$. It follows that $Y_1$ is nonempty since $(\varphi_1,\varphi_1) \in Y_1$.
We claim that there exists $C>0$ such that $$\|v\|_q^q \geq C\|(u,v)\|^q \quad \forall (u,v) \in \overline{Y}_1.$$ Otherwise we find a sequence $(u_n,v_n) \subset \overline{Y}_1$ such that
$\|v_n\|_q^q \leq \frac{1}{n}\|(u_n,v_n)\|^q$. We may assume that $(\bar{u}_n,\bar{v}_n):=\frac{(u_n,v_n)}{\|(u_n,v_n)\|} \rightharpoonup (u,v)$ in $X$.
It follows that $\|\bar{v}_n\|_q \to 0$ i.e. $v=0$ and from $(u_n,v_n) \subset \overline{Y}_1$ we infer that
$\|(u_n,v_n)\|^2 \le 2\lambda \int u_nv_n,$
i.e. $1 \le 2\lambda \int \bar{u}_n\bar{v}_n \to 0$, which is a contradiction.
Thus the claim is proved an since $\lambda<\lambda^*$, we see that (H5) holds. Finally, it is clear that $J$ satisfies (HJ) for any $u$. So Corollary \ref{cc} applies with
$$P(u,v)=\|v\|_q^q, \quad \Upsilon(u,v)=\|(u,v)\|^2-2\lambda\int_\Omega uv, \quad \mbox{and} \quad \Gamma(u,v)=-\int_\Omega b|u|^r.$$
\end{enumerate}
\end{proof}

\medskip

\subsection{A quasilinear gradient system}\strut \\

The previous results on  \eqref{pk3} can be easily extended to the quasilinear system
\begin{equation}\label{sp}
\left\{
\begin{array}{ll}
-\Delta_p u = \mu |u|^{p-2}u+\lambda \alpha|u|^{\alpha-2}u|v|^{\beta}+b(x)|u|^{r-2}u \
& \mbox{in} \ \ \Omega, \ \ \\
-\Delta_p v =  \delta |v|^{p-2}v +\lambda \beta |u|^{\alpha}|v|^{\beta-2}v-|v|^{q-2}v \
&\mbox{in} \ \ \Omega, \ \ \\
u = v = 0 \ &\mbox{on} \ \
\partial\Omega,
\end{array}
\right.
\end{equation}
where $\mu,\delta \in \Re$, $\lambda>0$, $1<p<q<r<p^*$, and $\alpha,\beta>1$ with $\alpha+\beta=p$.

The energy functional  is given by 
$$
\Phi(u,v)= \frac{1}{p}\int_\Omega\left( |\nabla u|^p-\mu |u|^p+|\nabla v|^p-\delta |v|^p\right)-\lambda\int_\Omega |u|^{\alpha}|v|^{\beta} 
+\frac{1}{q}\|v\|_q^q-\frac{1}{r}\int_\Omega b|u|^r,
$$
for $(u,v)\in X=W_0^{1,p}(\Omega) \times W_0^{1,p}(\Omega)$.

We set  $\|(u,v)\|=\left(\|\nabla u\|_{p}^p+\|\nabla v\|_{p}^p\right)^{\frac{1}{p}}$ and $\|(u,v)\|_p=\left(\| u\|_{p}^p+\| v\|_{p}^p\right)^{\frac{1}{p}}$.

Let 
\begin{equation*}
\lambda_1^*(\mu,\delta)=\inf\left\{\frac{\|\nabla u\|_p^p+\|\nabla v\|_p^p-\mu \|u\|_p^p-\delta\|v\|_p^p}{p\int_\Omega |u|^{\alpha}|v|^{\beta}}: (u,v)\in X, \int_\Omega |u|^{\alpha}|v|^{\beta}>0 \right\},
\end{equation*}
and
\begin{equation*}
\lambda^*(\mu,\delta):=\inf\left\{\lambda(u,v,\mu,\delta): \int_\Omega b|u|^r>0,\ \int_\Omega |u|^{\alpha}|v|^{\beta}>0 \right\},
\end{equation*}
where
$$
\lambda(u,v,\mu,\delta):=\frac{1}{\int_\Omega |u|^{\alpha}|v|^{\beta}} \left(\|\nabla u\|_p^p+\|\nabla v\|_p^p-\mu \|u\|_p^p-\delta\|v\|_p^p+ C_{p,r,q}\frac{\|v\|_q^{q\frac{r-p}{r-q}}}{\left(\int_\Omega b|u|^r \right)^{\frac{q-p}{r-q}}}\right),
$$ and
$C_{p,r,q}:=\frac{r-q}{r-p}\left(\frac{q-p}{r-p}\right)^{\frac{q-p}{r-q}}$.

It is not hard to show that $\lambda_1^*(\mu,\delta)$ and $\lambda^*(\mu,\delta)$ are achieved, so that $\lambda^*(\mu,\delta)>\lambda_1^*(\mu,\delta)$.
In addition, $\frac{\lambda_1(p)-\max(\mu,\delta)}{\max(\alpha,\beta)}\leq \lambda_1^*(\mu,\delta)\leq \frac{1}{p}(2\lambda_1(p)-\mu-\delta)$. Note that if $p=2$, $\mu=\delta$, and $\alpha=\beta=1$ then $\lambda_1^*(\mu,\delta)=\lambda_1-\mu$.

Let
$$Y_1=\left\{(u,v)\in X:\|\nabla u\|_p^p+\|\nabla v\|_p^p-\mu \|u\|_p^p-\delta\|v\|_p^p-p\lambda\int_\Omega |u|^{\alpha}|v|^{\beta}<0\right\}$$ and $$Y_2=\left\{(u,v)\in X: \int_\Omega b|u|^r> 0\right\}.$$

For  $\mu,\delta<\lambda_1(p)$ and $0<\lambda<\lambda_1^*(\mu,\delta)$ there exists $C>0$ such that 
\begin{equation}
\|\nabla u\|_p^p+\|\nabla v\|_p^p-\mu \|u\|_p^p-\delta\|v\|_p^p-p\lambda\int_\Omega |u|^{\alpha}|v|^{\beta}\ge C\|(u,v)\|^p
\end{equation}
for all $(u,v)\in X$ such that $\int_\Omega |u|^{\alpha}|v|^{\beta}>0$.
Indeed, otherwise there is a sequence $(u_n,v_n)$ such that $\|(u_n,v_n)\|=1$, $\int |u_n|^{\alpha}|v_n|^{\beta}>0$ and 
$$\limsup \left(\|\nabla u_n\|_p^p+\|\nabla v_n\|_p^p-\mu \|u_n\|_p^p-\delta\|v_n\|_p^p-p\lambda\int |u_n|^{\alpha}|v_n|^{\beta}\right)\le 0$$
We can assume that $(u_n,v_n) \rightharpoonup (u,v)$ in $X$, so that
$\|u\|^p+\|v\|^p-\mu \|u\|_p^p-\delta\|v\|_p^p-p\lambda \int |u|^{\alpha}|v|^{\beta}\le 0$
and $(u,v)\neq (0,0)$. Thus
$$0<\|\nabla u\|_p^p+\|\nabla v\|_p^p-\mu \|u\|_p^p-\delta\|v\|_p^p\le p\lambda \int_\Omega |u|^{\alpha}|v|^{\beta},$$
so that $\int_\Omega |u|^{\alpha}|v|^{\beta}>0$, which implies $\lambda \ge \lambda_1^*(\mu,\delta)$.

Now, if $\lambda_1^*(\mu,\delta)<\lambda<\lambda^*(\mu,\delta)$ then $Y_1$ is nonempty and arguing as in the proof of Corollary \ref{ws}(2) one can show that  for $\mu<\lambda_1(p)$ 
there exists $C>0$ such that $\|v\|_q^q \geq C\|(u,v)\|^q$ for any $(u,v) \in Y_1$.

By Corollary \ref{cc} we obtain the following result:
\begin{cor}\strut
	\begin{enumerate}
		\item If $\mu,\delta<\lambda_1(p)$ and $\lambda<\lambda_1^*(\mu,\delta)$ then \eqref{sp} has a ground state solution (at a positive level)
		\item If $\mu<\lambda_1(p)$ and $\lambda_1^*(\mu,\delta)<\lambda<\lambda^*(\mu,\delta)$ then \eqref{sp} has a ground state solution, which is a local minimizer (at a negative level), and a second critical point.
	\end{enumerate}	
\end{cor}
\medskip
\subsection{A  strongly coupled system and a fourth-order equation}\label{sh}\strut \\

Let us consider the Hamiltonian type system
$$
\left\{
\begin{array}{ll}
-\Delta u = |v|^{p-2}v \
& \mbox{in} \ \ \Omega, \ \ \\
-\Delta v = g(x,u) \
&\mbox{in} \ \ \Omega, \ \ \\
u = v = 0 \ &\mbox{on} \ \
\partial\Omega,
\end{array}
\right.\leqno{(S)}
$$
with $p > 1$. We shall apply our results to the fourth-order equation derived from $(S)$ by the {\it reduction by inversion} procedure, namely
$$
\left\{
\begin{array}{ll}
\Delta \left(|\Delta u|^{\frac{2-p}{p-1}}\Delta u\right) = g(x,u) \
&\mbox{in} \ \ \Omega, \ \ \\
u = \Delta u = 0 \ &\mbox{on} \ \
\partial\Omega.
\end{array}
\right.\leqno{(E)}
$$
The energy functional for $(E)$ is given by
$$
\Phi(u) = \dfrac{p-1}{p} \displaystyle\int_{\Omega} |\Delta u|^{\frac{p}{p-1}}   -
\displaystyle \int_{\Omega}G(x,u),
$$
which is a $C^1$ functional on
 $X = W^{2, \frac{p}{p-1}}(\Omega)\cap W_{0}^{1,\frac{p}{p-1}}(\Omega)$,
endowed with the norm
$ \|u\| = \left(\int_{\Omega} |\Delta u|^{\frac{p}{p-1}}  \right)^{\frac{p-1}{p}}$.

\begin{cor}\label{ch}
		Assume that $g:\Omega \times \Re \to \Re$ is a continuous map such that $g(x,t)\equiv 0$ for every $x \in \Omega \setminus \Omega'$, where $\Omega'$ is an open subset of $\Omega$. Moreover, assume that
	\begin{enumerate}
		
		\item $\dis \lim_{t \rightarrow 0} \frac{g(x,t)}{|t|^{\frac{1}{p-1}}}=0$ uniformly for $x \in \Omega'$.
		
		\item For every $x \in \Omega'$ the map $t \mapsto \frac{g(x,t)}{|t|^{\frac{1}{p-1}}}$
		is increasing on $\Re \setminus \{0\}$.
		\item $\dis \lim_{|t| \to \infty}\frac{G(x,t)}{|t|^{\frac{p}{p-1}}}=\infty$ uniformly for  $x \in \Omega'$.
	
		 \item If $N \geq 3$ and $p > \frac{N}{N -2}$ then there exists $C > 0$ and $q \in ( \frac{p}{p-1}, \sigma)$ such that
		 $$|g(x,t)| \leq C\left(1+|t|^{q-1}\right) \quad \forall \, t \in \mathbb{R},$$
 where  $\sigma = \frac{Np}{N(p-1) - 2p}$.
	\end{enumerate}
	Then $(E)$ has a positive ground state level.
\end{cor}

\begin{proof}
The energy functional is given by $\Phi=I_0-I$, where
$$I_0(u)=\dfrac{p-1}{p} \displaystyle \|u\|^{\frac{p}{p-1}}  \quad \mbox{and} \quad I(u)=\displaystyle \int_{\Omega}G(x,u) \quad \mbox{for } u \in X.$$
Let $\sigma=  \frac{Np}{N(p-1)-2p}$ if  $p>\frac{N}{N-2}$ and $\sigma=\infty$ if  $p=\frac{N}{N-2}$.
By (4) and the compact embeddings (see e.g. \cite{S})
\begin{equation}\label{x} \begin{cases} X \subset \mathcal{C}(\overline{\Omega}), & \text{ if  }N=2  \text{ and } p>1 \text{ or } N\geq 3 \text{ and } 1<p<\frac{N}{N-2}, \\  X \subset L^r(\Omega),  & \text{ if } p \geq \frac{N}{N-2} \text{ and } 1\leq r < \sigma,
\end{cases}
\end{equation}
one may show that $I'$ is strongly continuous. Proceeding as in the proof of Corollary \ref{c2} one may check that Corollary \ref{c1} applies with $Y:=\{u \in X: bu \not \equiv 0\}$.
\end{proof}

\begin{cor}
	Let $g(x,u)=b(x)|u|^{r-2}u$  with $b^+ \not \equiv 0$ and $r>\frac{p}{p-1}$ such that
	$\frac{1}{p}+\frac{1}{r}>\frac{N-2}{N}$.
	Then $(E)$ has a positive ground state level.
\end{cor}

\begin{proof}
Arguing as in the previous proof one may show that Corollary \ref{c1} applies with $Y:=\{u \in X: \int_\Omega b(x)|u|^r>0\}$.
\end{proof}

In the next result we deal with $\Lambda_1:=\displaystyle \inf_{u \in X \setminus \{0\}} \frac{\int_{\Omega} |\Delta u|^{\frac{p}{p-1}}}{\int_{\Omega} |u|^{\frac{p}{p-1}}}$,  the first eigenvalue of the problem
$$\Delta \left(|\Delta u|^{\frac{2-p}{p-1}}\Delta u\right) =\lambda |u|^{\frac{2-p}{p-1}}u \mbox{ in } \Omega, \quad u=\Delta u=0 \mbox{ on } \partial \Omega.$$
It is known that $\Lambda_1$ is simple \cite{DO}. We denote by $\psi_1$ a positive eigenfunction associated to $\Lambda_1$.
Let us set $$\Lambda^*:=\displaystyle \inf \left\{\frac{\int_{\Omega} |\Delta u|^{\frac{p}{p-1}}}{\int_{\Omega} |u|^{\frac{p}{p-1}}}: u \in X \setminus \{0\}, \int_\Omega b|u|^r \geq 0\right\}.$$
It is straightforward that $\Lambda^*>\Lambda_1$ if $\int_\Omega b\psi_1^r<0$.

\begin{cor}
	Let $g(x,u)=\lambda |u|^{\frac{2-p}{p-1}}u+  b(x)|u|^{r-2}u$  with $b^+ \not \equiv 0$ and $r>\frac{p}{p-1}$ such that
	$\frac{1}{p}+\frac{1}{r}>\frac{N-2}{N}$.
	\begin{enumerate}
\item 	If $\lambda<\Lambda_1$ then $(E)$ has a ground state solution at a positive level.
\item If $\int_\Omega b\psi_1^r<0$ and $\Lambda_1<\lambda<\Lambda^*$ then $(E)$ has a ground state solution at a negative level and a second solution.
	\end{enumerate}

\end{cor}

\begin{proof}
We have now $\Phi=I_0-I$, where
$$I_0(u) := \dfrac{p-1}{p} \displaystyle\int_{\Omega} \left(|\Delta u|^{\frac{p}{p-1}} -\lambda |u|^{\frac{p}{p-1}}\right) \quad \mbox{ and } \quad I(u):=
\displaystyle \int_{\Omega}b(x)|u|^r.$$
Set $Y_1:=\{u \in X: \int_\Omega b(x)|u|^r>0\}$.
One may easily show that for any $\lambda<\Lambda^*$ there exists $C>0$ such that $I_0'(u)u\geq C\|u\|^{\frac{p}{p-1}}$ for any $u \in \overline{Y}_1$. By the compact embedding \eqref{x} we have that $I'$ is strongly continuous. Note also that $I(u)=I'(u)u=0$ for $u \in \partial Y_1$, and since $r>\frac{p}{p-1}$ we have $I'(u)=o(\|u\|^{\frac{p}{p-1}})$. Finally, condition (4) of Corollary \ref{c1} holds with $\sigma=\frac{p}{p-1}$. Therefore $c_1:=\inf_{\mathcal{N}\cap Y_1} \Phi$ is positive and achieved by a critical point of $\Phi$.

 If $\lambda<\Lambda_1$ then $I_0'(u)u\geq C\|u\|^{\frac{p}{p-1}}$ for any $u \in X$ and since $I'(u)u\leq 0$ for any $u \in X \setminus Y_1$, we infer that $c_1$ is the ground state level of $\Phi$.

Now, if $\int_\Omega b\psi_1^r<0$ and $\Lambda_1<\lambda<\Lambda^*$ then 
$Y_2:=\{u \in X: I_0'(u)u<0\}$ is nonempty and there exists $C>0$ such that
$I'(u)u\geq C\|u\|^r$ for any $u \in \overline{Y}_2$. It follows that $\Phi$ is coercive in $Y_2$ and $\Phi \ge 0$ on $\partial Y_2$ and on $\mathcal{N} \setminus Y_2$. Thus $c_2:=\inf_{\mathcal{N}\cap Y_2} \Phi<0$ is the ground state level of $\Phi$ for $\Lambda_1<\lambda<\Lambda^*$.
\end{proof}


\medskip

\begin{thebibliography}{99}
	
	
	 \bibitem{AH} K. Adriouch, A. El Hamidi, The Nehari manifold for systems of nonlinear elliptic equations. Nonlinear Anal. 64 (2006), no. 10, 2149–2167
	
	
	
	
	
	
	
	\bibitem{BWW} T. Bartsch,  Z.-Q. Wang, Zhi-Qiang, M. Willem, The Dirichlet problem for superlinear elliptic equations. Stationary partial differential equations. Vol. II, 1–55, Handb. Differ. Equ., Elsevier/North-Holland, Amsterdam, 2005.
	
	
	
	
	\bibitem{BT} V. Bobkov, M. Tanaka,  {\it On positive solutions for $(p,q)$-Laplace equations with two parameters}. Calculus of Variations and Partial Differential Equations, 54(3), 3277-3301.
	
	\bibitem{BT2} V. Bobkov, M. Tanaka, {\it Remarks on minimizers for (p,q)-Laplace equations with two parameters.}
	Communications on Pure and Applied Analysis, 17(3), (2018) 1219-1253.
	
	\bibitem{BT3} V. Bobkov, M. Tanaka, {\it Multiplicity of positive solutions for $(p,q)$-Laplace equations with two parameters}. (2020), arXiv:2007.11623.
	
	\bibitem{BMR} D. Bonheure, E. Moreira dos Santos, and M. Ramos. Ground state and non-ground state solutions of some strongly coupled elliptic systems. Trans. Amer. Math. Soc., 364(1):447–491, 2012.
	
	\bibitem{BM} Y. Bozhkov,  E,  Mitidieri, Existence  of  multiple  solutions  for  quasilinear  systems  viafibering method, J. Differential Equations,190(2003), no. 1, 239–267. 
	

 	\bibitem{BW1} K.J. Brown, T-F. Wu, A semilinear elliptic system involving nonlinear boundary condition and sign-changing weight function. J. Math. Anal. Appl. 337 (2008), no. 2, 1326–1336.
	
	\bibitem{BW} K.J. Brown, T-F. Wu, A fibering map approach to a potential operator equation and its applications. Differential Integral Equations 22 (2009), no. 11-12, 1097–1114.
	
	

	
	\bibitem{CO} B. Chen, Z. Q. Ou,  Existence and bifurcation behavior of positive solutions for a class of Kirchhoff-type problems, Comput. Math. Appl. 77(10),  (2019) 2859--2866.
	
	\bibitem{CKW} C.-Y. Chen, Y.-C. Kuo, T.-F. Wu,  The Nehari manifold for a Kirchhoff type problem involving sign-changing weight functions. J. Differential Equations 250 (2011), no. 4, 1876–1908.
	
	\bibitem{CI} L. Cherfils and Y. Ilyasov, On the stationary solutions of generalized reaction diffusion equations with $p$\&$q$-Laplacian, Commun. Pure Appl. Anal. 4 (2005), no. 1, 9–-22.
	
	\bibitem{CW} S. Cingolani, T. Weth, On the planar Schrödinger-Poisson system, Ann. Inst. H. Poincaré Anal. Non Linéaire 33 (2016), 169-197.
	
	
	
	\bibitem{D} G. Dai, Eigenvalues, global bifurcation and positive solutions for a class of nonlocal elliptic equations, Topol. Methods Nonlinear Anal. 48(1), (2016) 213--233.
	
	
	
	
	
	\bibitem{DO} P. Drábek, M. Ôtani, Global bifurcation result for the p-biharmonic operator. Electron. J. Differential Equations 2001, No. 48, 19 pp.
	 
	\bibitem{FP} G. M. Figueiredo, M. T. O. Pimenta,  Nehari method for locally Lipschitz functionals with examples in problems in the space of bounded variation functions. NoDEA Nonlinear Differential Equations Appl. 25 (2018), no. 5, Paper No. 47, 18 pp.
	
	\bibitem{FRQ}G. M. Figueiredo, H. Ramos Quoirin, Ground states of elliptic problems involving nonhomogeneous operators, Indiana Univ. Math. J., 65(3), 2016, pg 779--795.
	
	
	
	
	
	\bibitem{I} Y. Il'yasov, On nonlocal existence results for elliptic equations with convex-concave nonlinearities, Nonlinear Anal. 61 (2005), no. 1-2, 211--236.
	
	\bibitem{I1} Y. Il'yasov, On extreme values of Nehari manifold method via nonlinear Rayleigh's quotient, Topol. Methods Nonlinear Anal. 49 (2017), no. 2, 683–714.
	
 \bibitem{LTW} Q. Li, K. Teng,  X. Wu, Ground states for Kirchhoff-type equations with critical growth. Commun. Pure Appl. Anal. 17 (2018), no. 6, 2623–2638.

\bibitem{LY} Q. Li, Z. Yang, Multiplicity of positive solutions for a p-q-Laplacian system with concave and critical nonlinearities. J. Math. Anal. Appl. 423 (2015), no. 1, 660–680. 
	
	 \bibitem{LWW} J. Liu, Y. Wang,  Z. Wang, Solutions for quasilinear Schrödinger equations via the Nehari method. Comm. Partial Differential Equations 29 (2004), no. 5-6, 879–901.
	
	\bibitem{MR} A. Molino, J. D. Rossi,  A concave-convex problem with a variable operator. Calc. Var. Partial Differential Equations 57 (2018), no. 1, Paper No. 10, 26 pp.
	
	\bibitem{E1} E. Moreira dos Santos, Multiplicity of solutions for a fourth-order quasilinear nonho- mogeneous equation. J. Math. Anal. Appl., 342(1):277–297, 2008.
	
	\bibitem{E2} E. Moreira dos Santos, On a fourth-order quasilinear elliptic equation of concave-convex type. NoDEA Nonlinear Differential Equations Appl., 16(3):297–326, 2009.
	
	
	 \bibitem{PRR} N. S. Papageorgiou, V. Rădulescu, D. Repovš,  Ground state and nodal solutions for a class of double phase problems. Z. Angew. Math. Phys. 71 (2020), no. 1, Paper No. 15, 15 pp.
	
	\bibitem{P} S. I. Pohozeav,	Nonlinear variational probelms via the fibering method
	Handbook of Differential Equations: Stationary Partial Differential Equations, Vol. 5, Elsevier (2008), pp. 49-209
	
	\bibitem{RQ} H. Ramos Quoirin, An indefinite type equation involving two p-Laplacians. J. Math. Anal. Appl. 387 (2012), no. 1, 189–200.
	
	\bibitem{S} A. Salvatore,  Multiple solutions for elliptic systems with nonlinearities of arbitrary growth. J. Differential Equations 244 (2008), no. 10, 2529–2544.

	\bibitem{SM} K. Silva, A. Macedo,  On the extremal parameters curve of a quasilinear elliptic system of differential equations. NoDEA Nonlinear Differential Equations Appl. 25 (2018), no. 4, Paper No. 36, 19 pp. 
	
	\bibitem{SS} K. Silva, S. M. Sousa, Finer analysis of the Nehari set associated to a class of Kirchhoff-type equations, SN Partial Differ. Equ. Appl. 1, 43 (2020).
	
		\bibitem{SS1} K. Silva, S. M. Sousa, Multiplicity of positive solutions for a gradient type cooperative/competitive elliptic system, Electron. J. Differential Equations 2020, Paper No. 10, 14 pp.
	
	
	\bibitem{SW} A. Szulkin and T. Weth, The method of Nehari manifold. Handbook of nonconvex analysis and applications, 597--632, Int. Press, Somerville, MA, 2010.
	
	\bibitem{T} M. Tanaka,  Generalized eigenvalue problems for(p,q)-Laplacian with indefinite weight. J. Math. Anal.Appl.419(2), 1181–1192 (2014).
	
	\bibitem{Ta} G. Tarantello, On nonhomogeneous elliptic equations involving critical Sobolev exponent.
	Ann. Inst. H. Poincar\'e Anal. Non Lin\'eaire 9 (1992), no. 3,
	281--304.
	
	
	 \bibitem{Ws} T.-F. Wu, The Nehari manifold for a semilinear elliptic system involving sign-changing weight functions. Nonlinear Anal. 68 (2008), no. 6, 1733–1745.
\end{thebibliography}
\end{document}